\documentclass[11 pt]{amsart}
\usepackage{tikz-cd}

\usepackage{amsmath,amsthm,amssymb}

\renewcommand{\min}{\mathrm{min}}

\newcommand{\cl}{\mathrm{cl}}
\newcommand{\R}{\mathbb{R}}
\newcommand{\dom}{\mathrm{dom}}

\newcommand{\N}{\mathbb{N}}

\newcommand{\iso}{\mathrm{Iso}}
\newcommand{\U}{\mathbb{U}}
\newcommand{\conv}{\mathrm{conv}}
\newcommand{\id}{\mathrm{id}}
\newcommand{\cone}{\mathrm{cone}}
\newcommand{\rng}{\mathrm{rng}}
\newcommand{\ext}{\mathrm{ext}}

\newcommand{\Iso}{\mathrm{Iso}}
\newcommand{\F}{F}
\newcommand{\diam}{\mathrm{diam}}
\newcommand{\ball}{\mathrm{ball}}

\newcommand{\Q}{\mathbb{Q}}
\newcommand{\Sch}{\mathrm{Sch}}
\newcommand{\Perf}{F_p}

\newcommand{\ring}{R}

\newtheorem{theorem}{Theorem}[section]
\newtheorem{lemma}[theorem]{Lemma}

\newtheorem{corollary}[theorem]{Corollary}
\newtheorem{proposition}[theorem]{Proposition}
\newtheorem{claim}[theorem]{Claim}

\theoremstyle{definition}
\newtheorem{definition}[theorem]{Definition}

\newtheorem{question}[theorem]{Question}


\begin{document}

\title[Completeness of isomorphism of C*-algebras]{Completeness of the isomorphism problem for
  separable C*-algebras}

\author{Marcin Sabok}\thanks{This research was partially
  supported by the NCN (the Polish National Science Centre)
  grant no. 2012/05/D/ST1/03206, by the Foundation for
  Polish Science and by MNiSW (the Polish Ministry of
  Science and Higher Education) grant no. 0435/IP3/2013/72.}

\address{Instytut Matematyczny Polskiej Akademii Nauk,
  ul. \'Sniadeckich 8, 00-956 Warszawa, Poland}
\address{Instytut Matematyczny Uniwersytetu Wroc\l awskiego,
  pl. Grunwaldzki 2/4, 50--384 Wroc\l aw, Poland}

\email{M.Sabok@impan.pl}

\begin{abstract}
  We prove that the isomorphism problem for separable
  nuclear C*-algebras is complete in the class of orbit
  equivalence relations. In fact, already the isomorphism of
  simple, separable AI C*-algebras is a complete orbit
  equivalence relation. This means that any isomorphism
  problem arsing from a continuous action of a separable
  completely metrizable group can be reduced to the
  isomorphism of simple, separable AI C*-algebras. As a
  consequence, we get that the isomorphism problems for
  separable nuclear C*-algebras and for separable
  C*-algebras have the same complexity. This answers
  questions posed by Elliott, Farah, Paulsen, Rosendal, Toms
  and T\"ornquist.
\end{abstract}

\maketitle

\section{Introduction}

Broadly speaking, a problem $P$ in a class $\Gamma$ is
called \textit{complete} in $\Gamma$ if any other problem in
$\Gamma$ can be reduced to $P$.  Complete problems typically
appear in logic and computer science, perhaps with the most
prominent examples of NP-complete problems.

In the continuous setting, a descriptive complexity theory
for problems arising as Borel and analytic equivalence
relations on standard Borel spaces, has been developed by
Kechris, Louveau, Hjorth and others
\cite{hkl,kl,kechris.scot,hjorth,hjorth.handbook} over the
last 30 years. The classification problems arising in this
setting are of the following form: given an analytic or
Borel equivalence relation $E$ on a standard Borel space
$X$, decide whether two points in $X$ are $E$-equivalent. Of
great interest here are the equivalence relations given by
Borel actions of separable completely metrizable (i.e.\
\textit{Polish}) groups on standard Borel spaces or,
equivalently, continuous actions of Polish groups on Polish
spaces (see \cite{becker.kechris}). Typically, isomorphism
problems arising in various areas of mathematics are easily
translated into this language. The relative complexity is
measured in terms of Borel reducibility: an equivalence
relation $E$ on $X$ is \textit{Borel reducible} to an
equivalence relation $F$ on $Y$ if there is a Borel map
$f:X\to Y$ such that $x_1\mathrel{E}x_2$ if and only if
$f(x_1)\mathrel{F}f(x_2)$ for every $x_1,x_2\in X$. The
meaning of this notion is that the function $f$, being
computable (Borel), gives a way of reducing the problem of
$E$-equivalence of points in $X$ to that of $F$-equivalence
of points in $Y$. We say that an equivalence relation $E$ is
\textit{complete} in a class $\Gamma$ of equivalence
relations if it belongs to $\Gamma$ and every relation $F$
in $\Gamma$ is Borel-reducible to $E$. For examples of
complete analytic equivalence relations see
\cite{lr,lrf}. Two relations $E$ and $F$ are
\textit{bi-reducible} if $E$ is Borel reducible to $F$ and
$F$ is Borel reducible to $E$. For group actions, also a
stronger notion (in the measure-theoretic context) is used:
two group actions on standard Borel measure spaces $X$ and
$Y$ are called \textit{orbit equivalent} if there is a Borel
isomorphism of $X$ and $Y$ which maps (a.e.) orbits to
orbits (see \cite{gaboriau}). We say that an equivalence
relation is an \textit{orbit equivalence relation} if it is
bi-reducible with an equivalence relation induced by a Borel
action of a Polish group. Descriptive complexity theory has
applications to various classification problems arising in
many areas and it has enjoyed spectacular successes, for
instance the striking result of Thomas \cite{thomas} on the
relative complexity of isomorphism problems for torsion-free
abelian groups or the results of Foreman, Rudolph and Weiss
\cite{foreman.rudolph.weiss} on the conjugacy problem in
ergodic theory.

The isomorphism problem for separable C*-algebras has been
studied since the work of Glimm in the 1960's and evolved
into the Elliott program that classifies C*-algebras via
their $K$-theoretic invariants. Glimm's result \cite{glimm},
restated in modern language, implies that the isomorphism
relation for UHF algebras is smooth (see \cite[Chapter
5.4]{gao}). In the 1970's the classification has been pushed
forward to AF algebras via the $K_0$ group
\cite{elliott.af}. The Elliott invariant, which consists of
the groups $K_0$ and $K_1$ together with the tracial simplex
and the pairing map, was conjectured (see
\cite{elliott.icm,elliott.toms}) to completely classify all
infinite-dimensional, separable, simple nuclear
C*-algebras. The conjecture has been verified for various
classes of C*-algebras, e.g certain classes of real rank
zero algebras, AH algebras of slow dimension growth or
separable, simple, purely infinite, nuclear algebras (modulo
the the universal coefficients theorem)
\cite{elliott.realrankzero,elliott.gong,
  elliott.gong.li,kirchberg.phillips,elliott.rordam,
  rordam.classification.I,kirchberg.rordam} and there have
been dramatic breakthroughs in the program, including the
counterexamples to the general classification conjecture
constructed by R\o rdam \cite{rordam} and Toms \cite{toms}.

Classification program of C*-algebras can be studied from
the point of view of descriptive complexity theory (cf
\cite{elliott.abstract}) and the framework has been set up
by Kechris \cite{kechris.alg} and Farah, Toms and
T\"ornquist \cite{ftt}. Next, generalizing the results of
Thomsen \cite{thomsen} and using Elliott's classification
for AI algebras \cite{elliott.ai}, Farah Toms and
T\"ornquist showed \cite[Corollary 5.2]{ftt} that the
isomorphism of separable simple AI algebras is an orbit
equivalence relation and is not classifiable by countable
structures (see \cite[Chapter 10]{gao}). Recently, the
former has been generalized by Elliott, Farah, Paulsen,
Rosendal, Toms and T\"ornquist \cite{efprtt}, who showed
that in fact, the isomorphism of separable C*-algebras is an
orbit equivalence relation. Farah, Toms and T\"ornquist
\cite{ftt} and Elliott, Farah, Paulsen, Rosendal, Toms and
T\"ornquist \cite{efprtt} asked what is the complexity of
the isomorphism problem of separable and separable nuclear
C*-algebras. Understanding the complexity of these problems
can shed new light to the whole classification problem. In
this paper we prove the following.

\begin{theorem}\label{main0}
  The isomorphism relation of separable $C^*$-algebras is
  complete in the class of orbit equivalence relations. In
  fact, already the isomorphism of separable nuclear (and
  even simple, separable, AI) C*-algebras is a complete
  orbit equivalence relation.
\end{theorem}

A perhaps less sophisticated way of stating this result is
to say that any possible classification of separable nuclear
C*-algebras must essentially use C*-algebras as the
invariants. On the other hand, one could say that, in a way,
Elliott's conjecture turns out to be true: any separable
C*-algebra is classified by the Elliott invariant, though of
a perhaps different algebra. This solves the following
problems \cite[Problem 9.3]{ftt}, \cite[Problem 9.7]{ftt},
\cite[Question 4.1]{efprtt}, \cite[Question 4.2]{efprtt}.

The proof of Theorem \ref{main0} must be geometric. In
\cite[Corollary 5.2]{ftt} Farah, Toms and T\"ornquist showed
that the relation of affine homeomorphism of Choquet
simplices is Borel reducible to the isomorphism relation of
simple, separable, AI algebras. Thus, Theorem \ref{main0}
follows from the following.

\begin{theorem}\label{main}
  The relation of affine homeomorphism of Choquet simplices
  is complete in the class of orbit equivalence relations.
\end{theorem}

Gao and Kechris and, independently, Clemens
\cite{clemens.gao.kechris} showed that the isometry relation
of separable complete metric spaces is a complete orbit
equivalence relation. Afterwards, using the theory of
Lipschitz free spaces \cite{arens.eells} of Weaver
\cite{weaver} and the results of Mayer-Wolf
\cite{mayer.wolf}, Melleray \cite{melleray.isometry} showed
that the isometry of separable Banach spaces is a complete
orbit equivalence relation. Our proof of Theorem \ref{main}
reduces the relation of isometry of separable complete
metric spaces to the affine homeomorphism of Choquet
simplices.

Haydon \cite{haydon} (see also \cite[Theorem 29.9]{choquet}
and \cite[Page 143]{nato}) showed that any Polish space is
homeomorphic to the extreme bondary of a simplex. One should
note, however, that this construction is by no means unique
and indeed, there are simplices with homeomorphic extreme
boundaries which are not affinely homeomorphic (see
\cite[Page 119]{alfsen.boundary}). The general problem of
determining a simplex from the structure of its extreme
boundary is called the Dirichlet extension problem and it
has a highly nontrivial solution \cite{alfsen.dirichlet}.
In this paper, we take a different approach and, given a
complete metric space, we construct a simplex whose extreme
boundary only contains the metric space as a subset. The
advantage over Haydon's construction is that this
construction is invariant under the isometry. Our simplices,
which we call $S$-extensions of metric spaces, seem to be
different from most of the typical constructions of
simplices appearing in the literature (which are either
formed as inverse limits or as quotients).

Our proof reveals also an interesting connection between
three categories of objects: separable metric spaces,
metrizable Choquet simplices and separable Banach spaces. It
is somewhat parallel to the connection between metric spaces
and Banach spaces given by the Arens--Eells
extensions. Recall that given a separable metric space, its
Arens--Eells extension is a separable Banach space and the
assignment of the Arens--Eells extensions is isometry
invariant. On the other hand, given a simplex, one can look
at the space of affine continuous functions on it. Now, the
$S$-extensions give an invariant assignment of separable
Banach spaces to separable metric spaces that factors
through simplices as follows:
\begin{center}
  \begin{tikzcd}
    \textrm{metric space} \arrow{rr}{S\textrm{-extension}} &
    & \textrm{Choquet simplex}\arrow{rr}{\textrm{affine
        space}} & & \textrm{Banach space}
  \end{tikzcd}
\end{center}
The three classes of separable metric spaces, metrizable
Choquet simplices and separable Banach spaces contain
universal objects that have been studied independently. The
Urysohn space $\U$ \cite{urysohn} is the universal
ultrahomogeneous separable metric space and the Urysohn
sphere $\U_1$ is its counterpart of diameter 1. Both these
spaces are constructed in the same way and have finite
isometry extension properties. For more on the Urysohn space
and sphere, see \cite[Chapter 5]{pestov} or
\cite{melleray1,melleray2,urysohn.workshop}. The
Gurari\u{\i} space \cite{gurarii} is a similar object in the
category of separable Banach spaces: a separable Banach
space with an almost isometric extension property. This
space turns out to be unique \cite{lusky.gurarii} (see also
\cite{solecki.kubis} for a recent proof of this result). The
Poulsen simplex \cite{poulsen} is the unique (see
\cite{lindenstrauss.olsen.sternfeld}) metrizable Choquet
simplex with dense extreme boundary. All these spaces can be
formed via Fra\"iss\'e constructions from finite (or
finite-dimensional) objects. It is known that the
Gurari\u{\i} space is (isomorphic to) the dual of the space
of affine functions on the Poulsen simplex
\cite{lusky.gurarii} (see also \cite{lusky.primary}) but the
relations between the Urysohn space and the Gurari\u{\i}
space or the Poulsen simplex remain unclear. Fonf and
Wojtaszczyk \cite{fonf.wojtaszczyk} showed that the
Gurari\u{\i} space is not isomorphic to the Arens--Eells
extension of the Urysohn space. It seems plausibile,
however, that a modification of the techniques of this paper
can produce an invariant assignment of simplices of metric
spaces so that the Poulsen simplex is associated to the
Urysohn space. Then, the dual affine space would be the
Gurari\u{\i} space.

\bigskip

This paper is organized as follows. In Section
\ref{sec:convex} we recall some basic facts from convex
analysis and Choquet's theory. In Section \ref{sec:isometry}
we recall and slightly strenghten the results of Clemens,
Gao and Kechris on the isometry of separable metric
spaces. Section \ref{sec:twisted} contains some elementary
back-and-forth constructions for building affine
homeomorphism. The $S$-extension construction appears in
Section \ref{sec:sextensions}. Sections \ref{sec:cones} and
\ref{sec:blowup} are a preparation for a coding construction
of metrics that appears in Section \ref{sec:final}. Section
\ref{sec:questions} contains some condluding remarks and
questions.

\subsection{Acknowledgement}

Part of this work was done during the author's stay at the
Fields Institute during the thematic program in the fall
2012. The author is grateful for the hospitality of the
Fields Institute and would like to thank Antonio Avil\'es,
George Elliott, Ilijas Farah and Stevo Todor\v cevi\'c for
many inspiring discussions.

\section{Convex analysis and Choquet's
  theory}\label{sec:convex}

For basic concepts of convex analysis and Choquet's theory
we refer the reader to \cite[Chapter 15]{handbook} and to
\cite{phelps,alfsen,asimow,choquet}. In this paper we use
the terms \textit{simplex} and \textit{Choquet simplex}
interchangably and all convex compact sets that we consider
are metrizable.

We consider the Hilbert cube $[0,1]^\N$ as a compact convex
subset of a locally convex topological vector space (e.g.
$\ell_\infty$ with the weak* topology) and with the standard
metric on $[0,1]^\N$ given by $d_{[0,1]^\N}(x,y)=\sum_n
2^{-n}|x(n)-y(n)|$. Given a set $A$ in a locally convex
vector topological space, we write $\conv(A)$ for the closed
convex hull of $A$. For a compact convex set $C$, we write
$\ext(C)$ for the set of extreme points of $C$. Given two
convex compact sets $C$ and $D$, we write $C\simeq D$ to
denote that $C$ and $D$ are affinely homeomorphic. We write
$\Delta^n$ for the $n$-dimensional simplex. 

Whenever we consider a metric on a topological space, we
assume it is compatible with the topology. Although convex
compact sets are typically considered only with the affine
and topological structure, we will sometimes use metrics on
convex compact sets in locally convex vector topological
spaces. We will always assume that if $C$ is a compact
convex set with a metric $d_C$ on it, then there is a bigger
convex compact set $D$ with $C-C\subseteq D$ and a norm
$||\cdot||$ on $D$ such that $d_C(x,y)=||x-y||$ for every
$x,y\in C$. For convex compact subsets of the Hilbert cube,
we can use the metric $d_{[0,1]^\N}$ and for subsets of $\R$
the standard distance on $\R$. Note that the open balls in
such metrics are convex.

Given an inverse system $(S_n,\pi_n)$ of simplices (with
$\pi_n:S_{n+1}\to S_n$ an affine continuous surjection) and
$x\in\varprojlim S_n$, we write $x\restriction S_i$ (or even
$x\restriction i$ if it does not cause confusion) for the
image of $x$ under the canonical projection map from
$\varprojlim S_n$ to $S_i$. Sometimes, we use the notation
$x\restriction S_i$ for $x\in S_j$ with $j>i$. It is worth
noting here that, in general, it is not true that any
continuous affine surjection is open and the exact
characterization of when this happens has been given by
Vesterstr\o m \cite{vesterstrom}. The inverse limit of a
system of simplices is again a simplex \cite{jellett} and
Lazar and Lindenstrauss \cite{lazar.lindenstrauss} proved
that any (metrizable) simplex is an inverse limit of a
sequence of finite-dimensional simplices. Recently,
L\'opez-Abad and Todor\v cevi\'c \cite{todorcevic.abad}
showed that a generic inverse limit of simplices is affinely
homeomorphic to the Poulsen simplex. For more on the Poulsen
simplex we refer the reader to \cite[Chapter 15]{handbook}
and \cite[Chapter 3 Section 7]{asimow}.

The analysis on convex compact sets is dual to the theory of
order unit Banach spaces via the spaces of affine functions
(see \cite[Chapter 2]{asimow} and \cite{effros}).

The standard Borel structure on the space of simplices was
introduced in \cite{ftt} and is based on the parametrization
of simplices proved by Lazar and Lindenstrauss
\cite{lazar.lindenstrauss} and the duality to the order unit
spaces. Alternately, one can also use the induced Borel
structure from the space of compact subsets of the Hilbert
cube (see \cite[Section 4.1.4]{ftt}) or of the Poulsen
simplex. Note that for a convex compact set $C$, the space
of compact convex subsets of $C$ is closed in $K(C)$. The
space of simplices contained in $C$ is not closed but it is
Borel in $K(C)$ \cite[Lemma 4.7]{ftt}.

We will need a simple observation about affine continuous
functions on convex compact sets. Note that if $(C,d_C)$ is
a compact convex set in a locally convex vector topological
space and $f,g:\Delta^n\rightarrow C$ are two affine
functions such that $d_C(f(e),g(e))\leq\varepsilon$ for some
$\varepsilon>0$ and every $e\in\ext(\Delta^n)$, then
$||f-g||_\infty\leq\varepsilon$.  This follows by a simple
induction on the dimension of the simplex (case $n=1$ is
obvious as the metric is given by a norm, and the induction
step follows from the fact that every point in $\Delta^n$
belongs to a one-dimensional simplex whose one vertex
belongs to $\Delta^{n-1}$ and the other is an extreme point
of $\Delta^n$). Note that since any finite-dimensional
convex compact set is an affine image of $\Delta^n$ for some
$n\in\N$, the above observation is true for any
finite-dimensional convex compact set in place of
$\Delta^n$. Now, since the convex span of extreme points is
dense in any convex compact set, the following is true as
well.

\begin{proposition}\label{affproximity}
  Let $C,D$ be convex compact sets in locally convex vector
  topological spaces and let $d_C$ be a metric on $C$. If
  $f,g:D\rightarrow C$ are two affine continuous functions
  such that $d_C(f(e),g(e))\leq\varepsilon$ for some
  $\varepsilon>0$ and every $e\in\ext(D)$, then
  $||f-g||_\infty\leq\varepsilon$
\end{proposition}

\section{Isometry of perfect Polish spaces of bounded
  diameter}\label{sec:isometry}

The standard Borel space of separable complete metric spaces
is identified with the space of closed subsets of the
Urysohn space $\U$ with its Effros Borel structure. Gao and
Kechris \cite{gao.kechris}, and independently Clemens
\cite{clemens,clemens.gao.kechris}, proved that the isometry
relation on the space of separable complete metric spaces is
complete in the class of orbit equivalence relations.

Similarly, we consider the standard Borel space of separable
complete metric spaces of diameter bounded by $1$ as the
space of closed subsets of the Urysohn sphere $\U_1$.

For a Polish space $X$, write $\Perf(X)$ for the space of
perfect (nonempty) subsets of $X$. Note if $X$ is compact,
then $\Perf(X)$ is $G_\delta$ in $K(X)$, which implies that
for every Polish space $X$ the set $\Perf(X)$ is Borel in
$F(X)$. Given this, the standard Borel space of perfect
separable metric spaces of diameter bounded by $1$ is
identified with $\Perf(\U_1)$.

The following result is essentially due to Gao, Kechris
\cite{gao.kechris} and Clemens
\cite{clemens,clemens.gao.kechris} and we only indicate the
necessary changes in their proof to obtain the stronger
statement.

\begin{proposition}\label{isomperfect}
  The isometry relation of perfect subspaces of $\U_1$ is
  complete in the class of orbit equivalence relations.
\end{proposition}
\begin{proof}
  First note that the isometry relation of perfect separable
  complete metric spaces is reducible to the isometry
  relation of perfect separable complete metric spaces of
  diameter bounded by $1$. This follows from the fact that
  two spaces $(X,d_X)$ and $(Y,d_Y)$ are isometric if and
  only if the spaces $(X,d_X\slash(1+d_X))$ and
  $(Y,d_Y\slash(1+d_Y))$ are isometric. Thus, it is enough
  to show that the isometry of perfect separable metric
  spaces is complete in the class of orbit equivalence
  relations.
 
  For a separable complete metric space $(X,d_X)$, define
  (see \cite[Section 2F]{gao.kechris}) the equivalence
  relation $E(X)$ on $F(X)^\N$ as the equivalence induced by
  the group $\iso(X)$ of isometries of $(X,d_X)$ acting on
  $F(X)^\N$ coordinatewise, i.e. for
  $(F_1,F_2,\ldots),(F_1',F_2',\ldots)\in F(X)$ we have
  $(F_1,F_2,\ldots)\mathrel{E(X)}(F_1',F_2',\ldots)$ if
  there is $g\in\iso(X)$ such that $g(F_i)=F_i'$ for each
  $i\in\N$. Gao and Kechris \cite[Section 2H]{gao.kechris}
  (cf. also \cite[Theorem 3.5.3]{becker.kechris}) show that
  every orbit equivalence relation is Borel reducible to
  $E(X_1)$ where $X_1$ is the completion of $\U^\infty$ and
  $\U^\infty$ is the union of the sequence
  $\U\subseteq\U^3\subseteq\U^6\ldots$ with the inclusion
  maps given by $\U^n\ni x\mapsto(x,x,x)\in\U^{3n}$. The
  Urysohn space $\U$ is perfect and hence so is $\U^\infty$
  and $X_1$.

  Put $X=X_1\cup\bigcup_n C_n$, where each $C_n$ is a copy
  of the Cantor space and the union of $C_n$'s is discrete
  (i.e. each $C_n$ is a clopen set in $X$). Now $X$ is a
  perfect Polish space. 

  We can actually assume that the diameter of $X_1$ is
  bounded by $1$ (by replacing $d_{X_1}$ with
  $d_{X_1}\slash(1+d_{X_1})$). Given a sequence
  $s=(F_1,F_2,\ldots)\in F(X_1)^\N$ define the metric $d_s$
  on $X$ in following way: $d_s\restriction X_1$ is the
  original metric $d_{X_1}$ on $X_1$ and $d_s\restriction
  C_n$ is the standard metric on the Cantor space (with
  diameter $1$), for $x\in C_n$ and $y\in C_m$ let
  $d_s(x,y)=|n-m|+1$ and for $y\in C_n$ and $x\in X$ we have
  $d_s(x,y)=(n+2)+d_{X_1}(x,F_n)$.

  It is routine \cite[Section 2G]{gao.kechris} to check that
  the map $s\mapsto (X,d_s)$ is Borel. Now, the map
  $s\mapsto d_s$ is a reduction of $E(X_1)$ to the isometry
  of perfect separable complete metric spaces. If $s,t\in
  F(X)^\N$ are $E(X_1)$-equivalent, then clearly $X_s$ and
  $X_t$ are isometric. On the other hand, suppose
  $\varphi:X_s\rightarrow X_t$ is an isometry. Then it must
  be the case that $\phi''X_1=X_1$ since the points in $X_1$
  are the only points $y\in X_s$ such that the set
  $\{d_s(x,y):x\in X_0,d_s(x,y)\in\N\}$ consists of all
  positive natural numbers, and the same is true for $y\in
  X_t$. Then, it easily follows that $\phi\restriction X_1$
  is an isometry of $X_1$ and this isometry witnesses the
  $E(X_1)$-equivalence of $s$ and $t$ (for details see
  \cite[Lemma 2.5]{gao.kechris}).
\end{proof}

The above proposition can be also proved by adapting the
proof of Clemens \cite{clemens,clemens.gao.kechris}. Gao and
Kechris \cite[Section 2D]{gao.kechris} show also that the
isometry relation of subspaces of $\U$ is bireducible with
the equivalence relation induced by the action of $\Iso(\U)$
on $F(\U)$. Gao and Kechris deduce it from the fact that for
every separable complete metric space $X$ there is an
isometric copy $Z(X)$ of $X$ in $\U$ such that $F(\U)\ni
X\mapsto Z(X)\in F(\U)$ is Borel and for every $X,Y\in
F(\U)$ if $X$ and $Y$ are isometric, then there is an
isometry $\varphi\in\iso(\U)$ such that $\varphi''
Z(X)=Z(Y)$. This follows from the fact (see Gromov
\cite[Page 79]{gromov} and \cite[Lemma 2.2]{gao.kechris})
that for every separable metric space there is a metric
extension $X^*$ of $X$ (obtained via the Kat\v etov
construction) that is canonically isometric to the Urysohn
space $\U$ and such that any isometry $\varphi: X\to Y$
extends to an isometry $\varphi^*:X^*\to Y^*$. For more
details on this construction see \cite[Page
325]{gao}. Exactly the same arguments apply to metric spaces
of diameter bounded by $1$ and the Urysohn sphere $\U_1$,
which gives the following.

\begin{proposition}\label{zsets}
  There is a Borel map $Z:\Perf(\U_1)\to\Perf(\U_1)$ such
  that $X$ is isometric to $Z(X)$ and if $X,Y\in\Perf(\U_1)$
  are isometric, then there is an isometry
  $\varphi\in\iso(\U_1)$ with $\varphi''Z(X)=Z(Y)$.
\end{proposition}

\section{Twisted homeomorphisms and approximate
  intertwinings}\label{sec:twisted}

Say that an inverse system $(S_n,\pi_n:n\in\N)$ of simplices
is \textit{increasing} if $S_n\subseteq S_{n+1}$ is a face
of $S_{n+1}$ for each $n\in\N$ and $\pi_{n+1}(s)=s$ for each
$s\in S_{n+1}$. Note that if $(S_n:n\in\N)$ is an increasing
system of simplices, then also $S_i\subseteq\varprojlim S_n$
is a face of $\varprojlim S_n$, for each $i\in\N$.

An \textit{approximate intertwining} between two increasing
inverse systems of simplices $(S_i:i\in\N)$ and
$(T_i:i\in\N)$ is a sequence of affine continuous injective
maps $\varphi_i:S_{n_i}\rightarrow T_{m_i}$ and
$\psi_i:T_{m_i}\rightarrow S_{n_{i+1}}$ (for some increasing
sequences $n_i$ and $m_i$ of natural numbers) such that
$\varphi_n\subseteq\psi_n{}^{-1}$,
$\psi_n\subseteq\varphi_{n+1}{}^{-1}$ (recall that the
systems are increasing) and for each $\varepsilon>0$ there
exists $i$ such that for each $j>i$ and for each $x\in
S_{m_j}$ and $y\in T_{m_j}$ we have
\begin{itemize}
\item $|d_{T_{m_i}}(\varphi_i(x\restriction
  S_{n_i}),\varphi_j(x)\restriction T_{m_i})|<\varepsilon$,
\item $|d_{S_{n_{i+1}}}(\psi_i(y\restriction
  T_{m_i}),\psi_j(y)\restriction S_{m_{i+1}})|<\varepsilon.$
\end{itemize}

\begin{proposition}\label{approxinter}
  Suppose $(S_i:i\in\N)$ and $(T_i:i\in\N)$ are two
  increasing inverse systems of simplices and there is an
  approximate intertwining between these systems. Then
  $\varprojlim S_i$ and $\varprojlim T_i$ are affinely
  homeomorphic via a map which extends all the maps in the
  approximate intertwining.
\end{proposition}
\begin{proof}
  Write $S=\varprojlim S_i$ and $T=\varprojlim T_i$. Let
  $(\varphi_n:n<\omega)$ and $(\psi_n:n<\omega)$ form an
  approximate intertwining and for simplicity assume that
  $n_i=m_i=i$.  Note that for each $x\in S$ and for each
  $n\in\N$ the sequence of points $(\varphi_k(x\restriction
  S_k)\restriction T_n:k\in\N)$ in $T_n$ is Cauchy, by the
  property of approximate intertwining. For $x\in S$ write
  $x_n=\lim_k \varphi_k(x\restriction S_k)\restriction
  T_n$. Note that given $n<m\in\N$ we have $x_m\restriction
  T_n=x_n$ since $(\lim_k \varphi_k(x\restriction
  S_k)\restriction T_m)\restriction T_n=\lim_k
  (\varphi_k(x\restriction S_k)\restriction T_m)\restriction
  T_n)=\lim_k \varphi_k(x\restriction S_k)\restriction
  T_n$. Thus, there exists a point $y\in T$ such that for
  each $y\restriction T_n=x_n$ for each $n\in\N$. Write
  $\varphi(x)=y$. Note that the map $\varphi$ is affine
  since $(\restriction T_n)\circ\varphi_n$ is affine for
  each $n\in\N$. Note also that for each $n\in\N$ we have
  $\varphi\restriction S_n=\varphi_n$ since if $x\in S_n$,
  then $\varphi_k(x\restriction S_k)=\varphi_n(x)$ for each
  $k>n$.
  \begin{claim}
    The map $\varphi$ is continuous.
  \end{claim}
  \begin{proof}
    Fix $\varepsilon>0$. We need $\delta>0$ such that if
    $x_1,x_1\in S$ are such that $d_S(x_1,x_2)<\delta$, then
    $d_T(\varphi(x_1),\varphi(x_2))<\varepsilon$.  Fix
    $n\in\N$ big enough so that $T_n$ is $\varepsilon\slash
    2$-dense in $T$ and $\varphi(x)$ is $\varepsilon\slash
    2$-close to $\varphi_n(x\restriction S_n)$. Let
    $\delta>0$ be such that if $d_{S_n}(x_1,x_2)<\delta$,
    then
    $d_{T_n}(\varphi_n(x_1),\varphi_n(x_2))<\varepsilon\slash
    2$. Then, clearly, $\delta$ is as needed.
  \end{proof}
  Analogously define the map $\psi:T\rightarrow S$ and argue
  that it is affine and continuous. Now, the facts that
  $\bigcup_n S_n$ is dense in $S$ and $\bigcup_n T_n$ is
  dense in $T$, $\varphi$ extends $\varphi_n$ and $\psi$
  extends $\psi_n$ for each $n\in\N$, imply that
  $\varphi^{-1}=\psi$, as $\varphi_n\subseteq\psi_n{}^{-1}$
  and $\psi_n\subseteq\varphi_{n+1}{}^{-1}$.
\end{proof}

Given a sequence $C_n$ of compact convex sets of a compact
convex set $C$ in a locally convex topological vector space,
say that $C_n$ is \textit{convergent} if it convergent in
$K(C)$ (i.e. in the Hausdorff metric) and write $\lim_n C_n$
for the limit. Note that any increasing sequence of compact
convex sets is convergent and then $\lim_n C_n$ is equal to
the closure of its union. Suppose that $C_n\subseteq X$ and
$D_n\subseteq X$ are two increasing sequences of compact
convex sets. Say that a sequence of affine homeomorphic
embeddings $\varphi_n:C_n\rightarrow C$ is a \textit{twisted
  approximation} if for each $\varepsilon>0$ there exists
$n_0$ such that for every $k,l>n_0$ we have
\begin{itemize}
\item[(i)] $d_C(\varphi_k(x),\varphi_l(x))<\varepsilon$ for
  every $x\in C_k$,
\item[(ii)] $d_C(\rng\varphi_k,\lim_n D_n)<\varepsilon$,
\item[(iii)] $\varphi_k$ is an $\varepsilon$-isometry.
\end{itemize}
In (ii) above $d_C$ stands for the Hausdorff metric
extending $d_C$.

\begin{proposition}\label{twisted}
  Suppose $(C_n:n\in\N)$ and $(D_n:i\in\N)$ are two
  increasing sequences of convex compact subets of a convex
  compact set $C$ in a locally convex topological vector
  space. If there is a twisted approximation from
  $(C_n:n\in\N)$ to $(D_n:n\in\N)$, then $\lim_n C_n$ and
  $\lim_n D_n$ are affinely homeomorphic.
\end{proposition}
\begin{proof}
  The proof is analogous to the proof of Proposition
  \ref{approxinter}. Note that by (i), for every
  $x\in\bigcup_n C_n$ the sequences $\varphi_n(x)$ is
  convergent. Thus, we can define $\varphi':\bigcup_n
  C_n\rightarrow C$ as $\varphi(x)=\lim_n\varphi_n(x)$. Note
  that (iii) implies that $\varphi'$ is continuous (even an
  isometry). Moreover, (iii) implies that given any
  $x\in\lim_n C_n$ and any sequence $x_k\in\bigcup C_k$
  convergent to $x$, the sequence $\varphi'(x_k)$ is
  convergent and the limit does not depend on the choice of
  the sequence $x_k$. Thus, $\varphi'$ extends uniquely to
  $\varphi:\lim_n C_n\to C$, which again is continuous. The
  condition (ii) implies that $\rng(\varphi)=\lim_n D_n$. By
  (iii) we also get that $\varphi$ is one-to-one and hence a
  homeomorphism. Finally, $\varphi'$ is affine as a limit of
  affine maps and hence so is $\varphi$.
\end{proof}

Typically, convex compact sets are considered up to affine
homeomorphism and that is why we stated Proposition
\ref{twisted} in the above form.  However, we will apply
twisted approximations to get actual equality of convex
compact sets. Say that a sequence of affine homeomorphic
embeddings $\varphi_n:C_n\rightarrow C$ is a \textit{strong
  twisted approximation} if it is a twisted approximation
and for each $\varepsilon>0$ there exists $n_0$ such that
for every $k>n_0$ we have
\begin{itemize}
\item[(i')] $d_C(\varphi_k(x),x)<\varepsilon$ for every
  $x\in C_k$
\end{itemize}

Constructing a strong twisted approximation thus boils down
to showing that the Hausdorff distance between $\lim_n C_n$
and $\lim_n D_n$ is zero.

\begin{proposition}\label{stwisted}
  Suppose $(C_n:n\in\N)$ and $(D_n:i\in\N)$ are two
  convergent sequences of convex compact subets of a convex
  compact set $C$ in a locally convex vector space. If there
  is a strong twisted approximation from $(C_n:n\in\N)$ to
  $(D_n:n\in\N)$, then $\lim_n C_n$ and $\lim_n D_n$ are
  equal.
\end{proposition}
\begin{proof}
  Note that the affine homeomorphism constructed in
  Proposition \ref{twisted} is actually the identity on
  $\bigcup_n C_n$ and hence on all of $\lim_n C_n$.
\end{proof}

\section{$S$-extensions of metric
  spaces}\label{sec:sextensions}

In this section we define the $S$-extensions of metric
spaces. The definitions below stem from a simple observation
that any finite metric space $(X,d_X)$ is isometric to the
extreme boundary of a convex compact subset of
$\ell_\infty^n$ for some $n\in\N$. Indeed, if
$X=\{x_1,\ldots,x_n\}$, then the map $x_i\mapsto
a_i=(d_X(x_i,x_1),\ldots,d_X(x_i,x_n))$ maps $X$
isometrically into the extreme boundary of the convex hull
of $\{a_i:i\leq n\}$ taken with the $\ell_\infty$ metric.

Given a metric space $(X,d_X)$, together with its
(enumerated) countable dense set $D\subseteq X$ and a
countable (enumerated) family $\F\subseteq C(X,[0,1])$ of
continuous functions with values in the unit interval, we
will form a convex compact set $S(X,d_X,D,\F)$. Let
$D=(d_n:n\in\N)$ and $\F=(f_n:n\in\N)$. Consider the vectors
$a_n=(f_1(d_n),f_2(d_n),\ldots)\in[0,1]^\N$ and let
$S_n(X,d_X,D,\F)$ be the convex hull of the set $\{a_i:i\leq
n\}$ in $[0,1]^\N$ Note that the sequence $S_n(X,d_X,D,\F)$
is increasing, hence convergent in $K([0,1]^\N)$.

\begin{definition}
  Given a metric space $(X,d_X)$ define $S(X,d_X,D,\F)$ as
  $\lim_{n\to\infty}S_n(X,d_X,D,\F)$.
\end{definition}

We will always consider $S(X,d_X,D,F)$ with the metric
$d_{[0,1]^\N}$ induced from the Hilbert cube. Note that
since all $S_n(X,d_X,D,\F)$'s are convex, so is
$S(X,d_X,D,\F)$. In general, this is all we know about this
set. We will show, however, that if the functions $\F$ are
carefully chosen, then $S(X,d_X,D,\F)$ is a simplex and its
set of extreme points contains a dense homeomorphic copy of
$X$.

For each $n$ write $S_n^n(X,d_X,D,\F)$ for the projection of
$S_n(X,d_X,D,\F)$ into the first $n$-many
coordinates. Treating $[0,1]^n$ as a subset of the Hilbert
cube, we see that $S(X,d_X,D,\F)$ is also the limit of the
sets $S_n^n(X,d_X,D,\F)$. 

\begin{definition}
  We say that a countable family $\F$ of Lipschitz 1
  functions on a metric space $(X,d_X)$ is
  \textit{saturated} if
  \begin{itemize}
  \item for every $x\in X$ the distance function $z\mapsto
    d_X(z,x)$ belongs to the uniform closure of $\F$
  \item for every $x_1,\ldots,x_n\in X$ there are
    $f_1,\ldots,f_n\in\F$ such that the vectors
    $(f_1(x_1),\ldots,f_n(x_1)),\ldots,(f_1(x_n),\ldots,f_n(x_n))\in\R^n$
    are linearly independent.
  \end{itemize}
\end{definition}

Note that for any metric space the set of distance functions
to a dense countable subset of the space satisfies the first
condition above. However, in general, the set of distance
functions does not satisfy the second condition above (with
a minimal example being a space having four elements).

We first show that the definition of $S(X,d_X,D,\F)$ does not
depend on the choice of the dense set $D$, in particular it
does not depend on the enumeration of $D$. From the point of
view of further applications, we would need the fact that if
$D$ and $E$ are two countable dense sets in $X$, then
$S(X,d_X,D,\F)$ and $S(X,d_X,E,\F)$ are affinely
homeomorphic. However, they are actually the same set.

\begin{lemma}\label{indep1}
  If $\F$ is saturated, and $D,E$ are two countable dense
  subsets of $X$, then $S(X,d_X,D,\F)$ and $S(X,d_X,E,\F)$
  are equal.
\end{lemma}
\begin{proof}
  Let $D=(d_n:n\in\N)$ and $E=(e_n:n\in\N)$. Write
  $b_n=(f_1(d_n),f_2(d_n),\ldots)\in[0,1]^\N$ and
  $c_n=(f_1(e_n),f_2(e_n),\ldots)\in[0,1]^\N$. Write $B_n$
  for $S_n(X,d_X,D,F)$ and $C_n$ for $S_n(X,d_X,E,F)$. Note
  that $\{b_n:n\in\N\}$ as well as $\{c_n:n\in\N\}$ are
  affinely independent. Now $S(X,d_X,D,\F)=\lim_n B_n$ and
  $S(X,d_X,E,\F)=\lim_n C_n$. We will construct a strong
  twisted approximation from $(B_n:n\in\N)$ to
  $(C_n:n\in\N)$.

  For each $n$ choose a subsequence $(k^n_m:m\in\N)$ such
  that $$d_X(d_n,e_{k^n_m})<1\slash m$$ for each
  $m\in\N$. For $m\in\N$, let $\varphi_m:B_m\rightarrow
  [0,1]^\N$ be the affine map which maps $b_i$ to
  $c_{k^i_m}$ for $i\leq n$. We claim that
  $(\varphi_k:k\in\N)$ is a strong twisted approximation.

  Note that $d_{[0,1]^\N}(a_i,\varphi_m(a_i))<1\slash m$ for
  each $m\in\N$ since the functions in $\F$ are Lipschitz
  1. By Proposition \ref{affproximity}, this implies (i')
  and (iii). To see (ii), note that $\rng\varphi_n\subseteq
  C$ for each $n$, so it is enough to show that for each
  $\varepsilon>0$ there is $n_0\in\N$ such that
  $\rng\varphi_n$ is $\varepsilon$-dense in $C$ for every
  $n>n_0$. Pick $\varepsilon>0$ and let $n_1\in\N$ be big
  enough so that $[0,1]^{n_1}$ is
  \mbox{$\varepsilon\slash4$}-dense in $[0,1]^\N$ (treat
  $[0,1]^n$ as a subset of $[0,1]^\N$ via the embedding
  $(x_1,\ldots,x_n)\mapsto(x_1,\ldots,x_n,0,0,\ldots)$). Write
  $C_i^{n_1}$ for the projection of $C_i$ to $[0,1]^{n_1}$
  and note that there is $n_2>n_1$ such that $C^{n_1}_{n_2}$
  is $\varepsilon\slash4$-dense in $C^{n_1}_n$ for every
  $n>n_2$. Thus, $C_{n_2}$ is $\varepsilon\slash 2$-dense in
  $\lim_n C_n$. For each $i\leq n_2$ pick $l_i$ so that
  $d_X(e_i,d_{l_i})<\varepsilon\slash 4$. Pick $n_0>n_2$ so
  that $n_0>l_i$ for each $i\leq n_2$ and $1\slash
  n_0<\varepsilon\slash 4$ and hence
  $$d_X(e_i,e_{k^{l_i}_{n_0}})<\varepsilon\slash
  2\quad\mbox{for each }i\leq n_2.$$ Then
  $$d_{[0,1]^\N}(c_i,\varphi_{n_0}(b_{l_i}))<\varepsilon\slash
  2\quad\mbox{for each }i\leq n_2,$$ which implies that
  $\rng\varphi_{n_0}$ is $\varepsilon$-dense in $B$, as well
  as is $\rng\varphi_n$ for every $n>n_0$ since the ranges
  increase. This ends the proof.
\end{proof}

In the sequel, we write $S(X,d_X,\F)$ rather than
$S(X,d_X,D,\F)$. Another way of stating the previous lemma
is then to say that $S(X,d_X,\F)$ is the closed convex hull
of the set $\{(f_1(x),f_2(x),\ldots):x\in X\}$. In
principle, this can be taken as the definition of
$S(X,d_X,F)$ but in further arguments we will use the
approximation of $S(X,d_X,F)$ by $S_n(X,d_X,D,F)$.

Note that $D$ can be seen as a subset of $S(X,d_X,D,\F)$ via
the map $d\mapsto(f_1(d),f_2(d),\ldots)$. Note also that
since all distance functions to the points of $D$ are in the
uniform closure of $\F$, the above map is an
embedding. Denote this map by $i^D_\F$.

\begin{lemma}\label{embedding}
  If $\F$ is saturated, then there is a canonical
  homeomorphic embedding $i_F$ of $X$ into $S(X,d_X,\F)$
  which extends $i_\F^D$ for all countable dense $D\subseteq
  X$ and is Lipschitz 1.
\end{lemma}
\begin{proof}
  Write $i_\F(x)=(f_1(x),f_2(x),\ldots)$ and note that
  $i_\F(x)\in S(X,d_X,\F)$ for every $x\in X$ by Lemma
  \ref{indep1}. It is clear that, when viewed as a map to
  $S(X,d_X,D,\F)$, $i_\F$ extends $i^D_\F$. The fact that
  $i_\F$ is an embedding of $X$ into $S(X,d_X,D,\F)$ follows
  from the fact that all distance functions to the points
  in $D$ are in the uniform closure of $\F$.

  To see that $i_\F$ is a homeomorphism, suppose that
  $x_n\in X$ and $x\in X$ are such that $i_\F(x_n)\to
  i_F(x)$. This means that $f(x_n)\to f(x)$ for every $f\in
  F$. Again, since the distance functions from all points in
  $X$ are in the uniform closure of $F$ we have $d(x_n,y)\to
  d(x,y)$ for every $y\in X$ and hence $x_n\to x$, which
  shows that $i_\F$ is a homeomorphism.

  Finally, the fact that $i_\F$ is Lipschitz 1 follows
  immediately from the fact that all functions in $\F$ are
  Lipschitz 1.
\end{proof}

In the sequel we will abuse notation and treat $X$ as a
subset of $S(X,d_X,\F)$, unless this can cause confusion.

\bigskip

Now, we need to locate the extreme points of the sets
$S(X,d_X,F)$. Notice that if $F_1\subseteq F_2$ are two
families of Lipschitz $1$ functions, then there is a natural
affine continuous projection from $S(X,d_X,F_2)$ to
$S(X,d_X,F_1)$ (which forgets the coordinates from
$F_2\setminus F_1$).

Note also that if $K$ and $L$ are convex compact sets and
$\varphi:K\rightarrow L$ is an affine continuous surjection,
then $\ext(L)\subseteq\varphi''\ext(K)$. This follows from
the fact that if $x\in L$ is an extreme point of $L$, then
$\varphi^{-1}(\{x\})$ is a compact convex set, so it
contains a relative extreme point, say $y$. Now $y$ must be
extreme in $K$ since otherwise $y$ is a nontrivial affine
combination of $y_1,y_2\in K$. Since $y$ was extreme in
$\varphi^{-1}(\{x\})$, these points cannot belong to
$\varphi^{-1}(\{x\})$ and hence
$\varphi(y_1)\not=\varphi(y_2)$, which contradicts the fact
that $x$ was extreme in $L$.

The above implies that the larger the family $F$ we take,
the more extreme points we get in $S(X,d_X,F)$. The lemma
below says that if $F$ is saturated (in fact here we only
need the first item from the definition), then we get quite
enough of them.

\begin{lemma}\label{extreme}
  If $\F$ is saturated, then the points of $X$ are extreme
  points in $S(X,d_X,\F)$.
\end{lemma}

Given a compact convex set $C$ in a locally convex vector
topological space, with a metric $d_C$ on $C$, and a subset
$A\subseteq C$ we say that $A$ is an
\textit{$\varepsilon$-face} provided that for every $x,y\in
C$ if $\frac{1}{2}(x+y)\in A$, then both $d_C(x,A)$ and
$d_C(y,A)$ are smaller than $\varepsilon$. Note that if
$A\subseteq C$ is written as $A=\bigcap_n A_n$ so that each
$A_n$ is closed convex and for each $\varepsilon>0$ there is
$n_0$ such that $A_n$ is an $\varepsilon$-face for each
$n>n_0$, then $A$ is a face.

\begin{proof}[Proof of Lemma \ref{extreme}]
  We will again look at $S(X,d_X,D,\F)$ for a fixed
  countable dense set $D\subseteq X$. By Proposition
  \ref{indep1}, it is enough to show that the points of $D$
  are extreme in $S(X,d_X,D,\F)$. Moreover, it is enough to
  show that given the enumeration of $D$ as
  $(d_1,d_2,\ldots)$, the point $d_1$ is extreme in
  $S(X,d_X,D,\F)$.

  \begin{claim}\label{convdiam}
    For any $A\subseteq X$ we have
    $\diam_{S(X,d_X,\F)}(\conv(i_\F A))\leq\diam_X(A)$.
  \end{claim}

  \begin{proof}
    This follows directly from the fact that $i_\F$ is
    Lipschitz 1 and from local convexity of the Hilbert
    cube.
  \end{proof}

  Write $A_n$ for the closed convex hull of the ball in $X$
  around $d_1$ of diameter $1\slash n$. Claim \ref{convdiam}
  implies that $\bigcap_n A_n$ contains only one point, and
  thus is the singleton $\{d_1\}$. We will show that for
  each $\varepsilon>0$ there is $n_0$ such that $A_n$ is an
  $\varepsilon$-face for every $n>n_0$.

  \begin{claim}\label{estimates}
    Suppose $\delta\geq0$ and $a_1,\ldots,a_k\in[0,1]$,
    $\alpha_1,\ldots,\alpha_k\in[0,1]$ are such that $\sum_i
    \alpha_i=1$. If $b\in[0,1]$ and $b_1,\ldots,b_k\in[0,1]$
    are such that $|b-b_i|\leq a_i+\delta$ for each $i\leq
    k$, then $$|b-\sum_{i=1}^k \alpha_i
    b_i|\leq\sum_{i=1}^k\alpha_i a_i+\delta.$$
  \end{claim}
  \begin{proof}
    This is just a straightforward computation. Note that we
    have $|b-\sum_i \alpha_i
    b_i|=|\sum_i\alpha_i(b-b_i)|\leq\sum_i\alpha_i|b-b_i|\leq\sum_i\alpha_i(
    a_i+\delta)=\sum_i\alpha_i a_i +\delta$.
  \end{proof}

  Fix $\varepsilon>0$ and find a function $f\in\F$ which is
  $\varepsilon\slash 8$-uniformly close to the distance
  function $z\mapsto d(z,d_1)$. Without loss of generality
  assume that $f=f_1$.

  Suppose now that $n>8\slash\varepsilon$ and $y,z\in
  S(X,d_X,\F)$ are such that $\frac{1}{2}(y+z)\in
  A_n$. Approximate $y$ with $y'$ and $z$ with $z'$ such
  that $d_{S(X,d_X,\F)}(y,y')<\varepsilon\slash 8$,
  $d_{S(X,d_X,\F)}(z,z')<\varepsilon\slash 8$ and $y',z'\in
  S_k(X,d_X,D,\F)$ for some $k\in\N$. Note that then
  $\frac{1}{2}(y'+z')$ is $\varepsilon\slash 8$-close to
  $A_n$. Since the first coordinate of every point in $A_n$
  is smaller than $1\slash n+\varepsilon\slash 8$, the first
  coordinate of $\frac{1}{2}(y'+z')$ is smaller than
  $1\slash n+\varepsilon\slash 4$. Hence, the first
  coordinates of both $y'$ and $z'$ are smaller than
  $2(1\slash n+\varepsilon\slash 4)$, which is smaller than
  $3\varepsilon\slash 4$. Since the first coordinate of
  points in $A_n$ is given by the function that is
  $\varepsilon\slash 8$-close to the the distance function
  from $d_1$ and all other coordinates are given by
  functions that extend Lipschitz 1 functions, Claim
  \ref{estimates} implies that
  $$d_{S(X,d_X,\F)}(y',d_1)<3\varepsilon\slash
  4+\varepsilon\slash 8\quad\mbox{as well as}\quad
  d_{S(X,d_X,\F)}(z',d_1)<3\varepsilon\slash 4+\varepsilon\slash
  8.$$ This implies that
  $d_{S(X,d_X,\F)}(y,d_1)<3\varepsilon\slash
  4+\varepsilon\slash 4=\varepsilon$ and
  $d_{S(X,d_X,\F)}(z,d_1)<3\varepsilon\slash
  4+\varepsilon\slash 4=\varepsilon$, which shows that $A_n$
  is an $\varepsilon$-face.
\end{proof}

The fact that $X$ is dense in the set of extreme points of
$S(X,d_X,F)$ will follow from the following general lemma.

\begin{lemma}
  Suppose $C_n$ is an increasing sequence of compact convex
  subsets of a metrizable convex compact set $C$ in a
  locally convex topological vector space. Then
  $\bigcup_n\ext(C_n)$ is dense in $\ext(\lim_n C_n)$.
\end{lemma}
\begin{proof}
  Write $K$ for the closure of $\bigcup_n\ext(C_n)$ and
  suppose $x\in\lim_n C_n$ is such that $x\notin K$. We need
  to show that $x$ is not an extreme point of $\lim_n
  C_n$. Note that \cite[Chapter 15, Proposition
  2.3]{handbook} $x$ is a bacycenter of a probability
  measure $\mu$ concentrated on $K$. By a theorem of Bauer
  \cite{bauer} (see also \cite[Proposition 1.4]{phelps}) if
  $x$ is an extreme point and a barycenter of a probability
  measure $\mu$, then $\mu=\delta_x$. But $\mu$ is
  concetranted on $K$, so cannot be equal to $\delta_x$.
\end{proof}

This immediately gives the following.

\begin{corollary}\label{density}
  If $\F$ is saturated, then $X$ is dense in the set of extreme
  points of $S(X,d_X,\F)$.
\end{corollary}

Let us now see some examples of convex compact sets that can
arise as $S(X,d_X,\F)$. The next proposition will not be
used later in the proof but it shows some ideas behind the
constructions that follow.

In the following proposition, for a compact metric space
$X$, write $P(X)$ for the Bauer simplex of all Borel
probability measures on $X$. The extreme boundary of $P(X)$
is canonically homeomorphic to $X$ and any simplex whose
extreme boundary is homeomorphic to $X$ is canonically
affinely homeomorphic to $P(X)$ (this follows for example
from the positive solution to the Dirichlet extension
problem for Bauer simplices \cite{alfsen.dirichlet}).

As a comment to the assumption of the following proposition,
note that if $(X,d_X)$ is compact, then Lipschitz functions
are dense in $C(X)$, by the Stone--Weierstrass
theorem. Thus, if $(X,d_X)$ is compact, then the set of all
Lipschitz 1 functions is linearly dense in $C(X)$.

\begin{proposition}\label{bauer}
  Suppose $(X,d_X)$ is compact, $\F$ is saturated and
  linearly dense in $C(X)$. Then $S(X,d_X,\F)$ is affinely
  homeomorphic to the Bauer simplex $P(X)$.
\end{proposition}
\begin{proof}
  Note that the set of extreme points of $S(X,d_X,\F)$ is
  equal to $X$ by Lemma \ref{extreme} and Corollary
  \ref{density}. We need to prove that $S(X,d_X,\F)$ is a
  simplex. For that, suppose $\mu,\nu$ are two distinct
  probability measures on $X$. By linear density of $\F$ in
  $C(X)$, there is $f\in \F$ such that $\int fd\mu\not=\int
  fd\nu$. Without loss of generality, assume that
  $f=f_1$. Pick a countable dense set $D\subseteq X$ and
  choose two sequences of atomic measures $\mu_n$ and
  $\nu_n$ concentrated on $D$ such that $\mu_n\to\mu$ and
  $\nu_n\to\nu$. Note that each $\mu_n$ and $\nu_n$ has a
  barycenter in one of the sets $S_k(X,d_X,D,\F)$, for if
  $$\mu_n=\sum_{i=1}^{k_n}\alpha^n_i
  \delta_{d^n_i}\quad\mbox{with}\quad d^n_i\in D,
  \alpha^n_i\geq 0,
  \sum_{i=1}^{k_n}\alpha^n_i=1,$$ $$\nu_n=\sum_{i=1}^{l_n}\beta^n_i
  \delta_{e^n_i}\quad\mbox{with}\quad e^n_i\in D,
  \beta^n_i\geq 0, \sum_{i=1}^{k_n}\beta^n_i=1,$$ then the
  barycenter of $\mu_n$ is $\sum_{i=1}^{k_n}\alpha^n_i
  d^n_i$ and the barycenter of $\nu_n$ is
  $\sum_{i=1}^{l_n}\beta^n_i e^n_i$. Note that then
  $\sum_{i=1}^{k_n}\alpha^n_i f_1(d^n_i)$ and
  $\sum_{i=1}^{l_n}\beta^n_i f_1(e^n_i)$ are the first
  coordinates in $[0,1]^\N$ of the barycenters of $\mu_n$
  and $\nu_n$. Thus, the first coordinate of the barycenter
  of $\mu$ is the limit of the first coordinates of the
  barycenters of $\mu_n$ \cite[Chapter 15, Proposition
  2.2]{handbook} and is equal
  to $$\lim_n\sum_{i=1}^{k_n}\alpha^n_i f_1(d^n_i)=\lim_n
  \int f_1d\mu_n=\int f_1 d\mu$$ and the first coordinate of
  the barycenter of $\nu$ is the limit of the first
  coordinates of the barycenters of $\nu_n$ and is equal
  to $$\lim_n\sum_{i=1}^{l_n}\beta^n_i f_1(e^n_i)=\lim_n
  \int f_1d\nu_n=\int f_1 d\nu.$$ Thus, $\mu$ and $\nu$ have
  different barycenters, which shows that $S(X,d_X,\F)$ is a
  simplex.
\end{proof}

Now we need to see how $S(X,d_X,\F)$ depends on the choice of
$F$. First note that it does not depend on the enumeration
of $F$.

\begin{lemma}\label{permut}
  Suppose $F$ is saturated and $G$ enumerates the same set of
  functions as $F$. Then $S(X,d_X,\F)$ and $S(X,d_X,G)$ are
  affinely homeomorphic.
\end{lemma}
\begin{proof}
  Let $\pi:\N\to\N$ be a permutation such that
  $G=(f_{\pi(1)},f_{\pi(2)},\ldots)$. Let
  $h:[0,1]^\N\to[0,1]^\N$ be defined as
  $h(x_1,x_2,\ldots)=(x_{\pi(1)},x_{\pi(2)},\ldots)$ and
  note that $h$ is an affine homeomorphism of the Hilbert
  cube which maps $S(X,d_X,\F)$ to $S(X,d_X,G)$.
\end{proof}

Suppose now that $(X,d_X)$ is a metric space, $D\subseteq X$
is a countable dense set and that $F$ consists of distance
functions from the points in $D$. In such a case, we denote
$F$ by $d_XD$. Note that $d_XD$ always consists of Lipschitz
1 functions and its uniform closure contains all distance
functions from the points in $X$. In general, still, $d_XD$
need not be saturated. We will see, however, that if $X$ has
certain additional property (e.g. is the Urysohn space or
the Urysohn sphere), then this is the case.

\begin{definition}
  We say that that a separable metric space $(X,d_X)$ is
  \textit{saturated} if $d_XD$ is saturated for every
  countable dense set $D\subseteq X$.
\end{definition}

\begin{proposition}\label{indep3}
  Suppose $(X,d_X)$ is separable, complete and saturated and
  $D,E\subseteq X$ are two countable dense sets. Then there
  is an affine homeomorphism $\tau_D^E:S(X,d_X,d_XD)\to
  S(X,d_X,d_XE)$ such that the following diagram commutes.
  \begin{center}
    \begin{tikzcd}
      S(X,d_X,d_XD) \arrow{rr}{\tau_D^E} & & S(X,d_X,d_XE)\\
      & X \arrow{ul}{i_{d_XD}}\arrow{ur}[below]{\quad\quad i_{d_XE}} & \\
    \end{tikzcd}
  \end{center}
\end{proposition}
The map $\tau_D^E$ should be treated as the transition map
between the coordinate systems of $D$ and $E$.
\begin{proof}
  First note that isolated points of $X$ must belong to both
  $D$ and $E$ and hence if $X_1$ is the set of isolated
  points, then by Lemma \ref{permut} we can assume that
  $X_1$ is enumerated in the same way in $D$ and $E$. Thus,
  without loss of generality we can assume that $X$ has no
  isolated points. Moreover, we can assume that $D$ and $E$
  are disjoint since we can always use a third countable
  dense set which is disjoint from both $D$ and $E$. Write
  $D=(d_0,d_1,\ldots)$ and let $D'=(d_1,d_2,\ldots)$.
  \begin{claim}\label{subst}
    $S(X,d_X,d_XD)$ and $S(X,d_X,d_XD')$ are affinely
    homeomorphic via a map $\tau_D^{D'}:S(X,d_X,d_XD)\to
    S(X,d_X,d_XD')$ such that $\tau_D^{D'}\circ
    i_{d_XD}=i_{d_XD'}$.
  \end{claim}
  \begin{proof}
    Write $C_n=S_n(X,d_X,D,d_XD)$ and
    $B_n=S_n(X,d_X,D,d_XD')$ for each $n\in\N$. Write also
    $[0,1]^\N$ as $[0,1]^{\{2,3,\ldots\}}\times[0,1]$
    interpreting the second coordinate of the product as the
    first coordinate in the Hilbert cube. Pick a subsequence
    $d_{k_n}\to d_0$ in $D$ and note that the distance
    function to $d_0$ is the uniform limit of the distance
    functions to $d_{k_n}$'s. By Proposition
    \ref{affproximity}, this implies that on $\bigcup_n C_n$
    the first coordinate is the uniform limit of the
    coordinates numbered with $k_n$'s. Thus, $\bigcup_n C_n$
    treated as a subset of
    $[0,1]^{\{2,3,\ldots\}}\times[0,1]$, is a graph of an
    affine function, say $a$, defined on $\bigcup_n
    B_n$. Moreover, since $a$ is a uniform limit of the
    functions given by coordinates in $\{k_1,k_2,\ldots\}$,
    and these functions clearly extend to the closure of
    $\bigcup_n B_n$, the function $a$ also uniquely extends
    to an affine function defined on the closure of
    $\bigcup_n B_n$, which is equal to
    $S(X,d_X,d_XD')$. Note that the graph of the unique
    extension is equal to the closure of $\bigcup_n C_n$,
    i.e.  $S(X,d_X,d_XD)$. Given that, the projection
    function from $[0,1]^{\{2,3,\ldots\}}\times[0,1]$ to
    $[0,1]^{\{2,3,\ldots\}}$ is an affine isomorphism from
    $S(X,d_X,d_XD)$ to $S(X,d_X,d_XD')$. Write $\tau_D^{D'}$
    for this isomorphism and note that since it just erases
    the first coordinate, we have $\tau_D^{D'}\circ
    i_{d_XD}=i_{d_XD'}$.
  \end{proof}
  Write now $D_n$ for the sequence
  $(e_1,e_2,\ldots,e_n,d_{n+1},d_{n+1},\ldots)$.
  \begin{claim}
    For each $n\in\N$, there is an affine homeomorphism
    $\tau_n:S(X,d_X,d_XD)\to S(X,d_X,d_XD_n)$ such that
    $\tau_n\circ i_{d_XD}=i_{d_XD_n}$.
  \end{claim}
  \begin{proof}
    It is enough to show that there is an affine
    homeomorphism $\tau_n':S(X,d_X,d_XD_n)\to
    S(X,d_X,d_XD_{n+1})$ with $\tau_n'\circ
    i_{D_n}=i_{D_{n+1}}$ and without loss of generality
    assume that $n=0$. But this follows from Claim
    \ref{subst} since both $S(X,d_X,d_XD)$ and
    $S(X,d_X,d_XD_1)$ are affinely homeomorphic to
    $S(X,d_X,d_XD')$ via maps which make the appropriate
    diagrams commute.
  \end{proof}

  Now note that since $\tau_{n+1}\tau_n^{-1}$ changes only
  the $n$-th coordinate, the sequence $\tau_n$ is uniformly
  Cauchy with respect to the Hilbert cube metric, and thus
  converges to an affine map
  $\tau_D^E:S(X,d_X,d_XD)\to[0,1]^\N$. Moreover, since the
  distance of $\rng\tau_n$ to $B_n$ is smaller than
  $2^{-n}$, we have that $\tau_D^E$ is an affine map from
  $S(X,d_X,d_XD)$ to $S(X,d_X,d_XE)$. In the same way
  construct maps $\tau^n:S(X,d_X,d_XE)\to S(X,d_X,d_XE_n)$
  with $E_n=(d_1,d_2,\ldots,d_n,e_{n+1},e_{n+2},\ldots)$ and
  note that their limit is an affine function
  $\tau_E^D:S(X,d_X,d_XE)\to S(X,d_X,d_XD)$. The maps are
  clearly inverse to each other, so $\tau_D^E$ is an affine
  homeomorphism. Finally, $\tau_D^E\circ i_{d_XD}=i_{d_XE}$
  follows from the fact that $\tau_n\circ
  i_{d_XD}=i_{d_XD_n}$ holds for each $n$ and $i_{d_XD_n}\to
  i_{d_XE}$. This ends the proof.
\end{proof}

The following proposition is stated for the Urysohn sphere
$\U_1$ but it also holds true (with the same proof) for the
Urysohn space $\U$. We say that a function
$f:X\to[0,\infty]$ defined on a metric space $(X,d_X)$ is a
\textit{Kat\v etov function} if for every $x,y\in X$ we have
$|f(x)-f(y)|\leq d_X(x,y)\leq f(x)+f(y)$ (cf. \cite[Lemma
5.1.22]{pestov} and \cite[Definition 1.2.1]{gao}).

\begin{proposition}\label{usaturation}
  The Urysohn sphere $\U_1$ is saturated.
\end{proposition}
\begin{proof}
  Write $d$ for the metric $d_{\U_1}$ on $\U_1$. Pick a
  countable dense set $D\subseteq \U_1$ and let
  $x_1,\ldots,x_n\in\U_1$. We show that there are
  $d_1,\ldots,d_n\in D$ such that the matrix
  \[ \left( \begin{array}{ccc}
      d(x_1,d_1) & \ldots & d(x_1,d_n) \\
      \vdots & \ddots & \vdots  \\
      d(x_n,d_1) & \ldots & d(x_n,d_n)
    \end{array} \right)\]
  is invertible. The proof is by induction. For $n=1$, any
  $d_1\not=x_1$ will do. Suppose $n>1$ and
  $x_1,\ldots,x_n\in\U$ are given. Pick any
  $d_1,\ldots,d_{n-1}$ that witness the inductive assumption
  for $x_1,\ldots,x_{n-1}$ and let $d\in D$ be such that
  $d(x_n,d)<d(x_i,d)$ for all $i\not=n$.. Consider
  the function 
  \[ \varepsilon\mapsto\det\left( \begin{array}{cccc}
      d(x_1,d_1) & \ldots & d(x_1,d_{n-1}) & d(x_1,d) \\
      \vdots & \ddots & \vdots & \vdots \\
      d(x_n,d_1) & \ldots & d(x_n,d_{n-1}) & d(x_n,d)+\varepsilon
    \end{array} \right)\]
  and note that it is a nonzero linear function, so there are
  arbitrarily small
  $\varepsilon>0$ at which it does not vanish. Pick such
  $\varepsilon_0>0$ which is smaller than
  $\min\{d(x_i,d)-d(x_n,d): i< n\}$. Note now that the
  function $f:\{x_1,\ldots,x_n\}\to[0,1]$ given by
  $f(x_i)=d(x_i,d)$ if $i<n$ and
  $f(x_n)=d(x_n,d)+\varepsilon_0$ is a Kat\v etov
  function with values in $[0,1]$. Thus, there is $y\in\U_1$ which realizes
  $f$. This means that 
  \[ \det\left( \begin{array}{cccc}
      d(x_1,d_1) & \ldots & d(x_1,d_{n-1}) & d(x_1,y) \\
      \vdots & \ddots & \vdots & \vdots  \\
      d(x_n,d_1) & \ldots & d(x_n,d_{n-1}) & d(x_n,y)
    \end{array} \right)>0\]
  Since the set of such $y$'s is clearly open, we can find
  one, say $d_n$, in $D$.
\end{proof}

\medskip

The Urysohn space has the extension property saying that
every Kat\v etov function defined on its finite subset is
realized as a distance function to some point in the
space. Huhunai\v svili \cite{huhu} (cf also
\cite{gromov,joiner,bogatyi,melleray1,melleray2}) showed
that the same is true for Kat\v etov functions defined on
compact subsets of the Urysohn space. This is probably the
strongest result that can guarantee that certain Kat\v etov
functions can be realized in the Urysohn space (or the
Urysohn sphere). The following construction is motivated by
the need of realizing Kat\v etov functions defined on
non-compact subspaces of the Urysohn sphere.

\medskip

Let $(X,d_X)$ be a separable metric space of diameter
bounded by 1 and let $D\subseteq X$ be its dense countable
subset. Let $R(D)$ be the ring of functions generated by the
distance functions $x\mapsto d_X(x,d)$ for $d\in D$ and all
rational constant functions. Note that $R(D)$ is countable
and all functions in $R(D)$ are Lipschitz. Let $R_1(D)$ the
the family of functions in $R(D)$ which are Lipschitz 1 and
have the range contained in $[\frac{1}{2},1]$. Note that
dividing a bounded Lipschitz function by an appropriately
large constant, we get a Lipschitz 1 function with the range
contained in $[-\frac{1}{4},\frac{1}{4}]$. Next, adding
$\frac{3}{4}$ we get a Lipschitz 1 function with
$\rng(f)\subseteq[\frac{1}{2},1]$. This shows that $R_1(D)$
is linearly dense in $R(D)$. On the other hand, since $d_X$
is bounded by 1, every function in $R_1(D)$ is a Kat\v etov
function. Recall that the elements of $X$ are identified
with Kat\v etov functions on $X$ by the Kuratowski
construction \cite[Chapter 1.2]{gao}, i.e. $x\in X$ is
identified with the function $y\mapsto d_X(x,y)$. Write
$F(X)$ for the family of all finitely supported (see
\cite[Definition 1.2.2]{gao}) Kat\v etov functions on $X$
with values in $[0,1]$. Let $E(X,d_X,D)$ be the completion
of the space $F(X)\cup R_1(D)$ with the sup metric. Note
that $E(X,d_X,D)$ is an extension of $X$ (as $F(X)$ contains
all functions $y\mapsto d_X(x,y)$ for $x\in X$), is
separable and realizes all finitely supported Kat\v etov
functions on $X$. However, it is slightly bigger than the
usual one-step Kat\v etov extension since we also have
realized the functions in $R_1(D)$. A standard argument
shows that if $\varphi:X\to X$ is an isometry, then
$\varphi$ extends to an isometry $\varphi'$ of $E(X,d_X,D)$
and the definition of $E(X,d_X,D)$ does not depend on the
choice of the dense set $D$. Thus, slightly abusing
notation, we write $E(X,d_X)$ for $E(X,d_X,D)$. Note,
however, that given $D$, we have a canonical countable dense
set $D'$ in $E(X,d_X,D)$ consisting of $R_1(D)$ and all
Kat\v etov functions finitely supported on a subset of $D$
and assuming rational values on their support.

Now, similarly as in the Kat\v etov construction of the
Urysohn space, we iterate the above extension construction
infinitely many times.

\begin{definition}
  Given a separable metric space $(X,d_X)$ and its countable
  dense subset $D\subseteq X$ define inductively
  $E^n(X,d_X,D)$ and $D^n(X,d_X,D)\subseteq E^n(X,d_X,D)$ as
  follows. $E^0(X,d_X,D)=(X,d_X)$ and
  $D_0(X,d_X,D)=D$. Given $E^n(X,d_X,D)$ and $D^n(X,d_X,D)$,
  which is a countable dense subset of $E^n(X,d_X,D)$ let
  $E^{n+1}(X,d_X,D)=E(E^n(X,d_X,D),D^n(X,d_X,D))$ and let
  $D^{n+1}(X,d_X,D)$ be the canonical extension of
  $D^n(X,d_X,D)$ to a countable dense subset of
  $E^{n+1}(X,d_X,D)$. Write $E^\infty(X,d_X,D)$ for the
  completion of the space $\bigcup E^n(X,d_X,D)$ and
  $D^\infty(X,d_X,D)$ for $\bigcup_n D^n(X,d_X,D)$.
\end{definition}

Again, the construction of $E^\infty(X,d_X,D)$ does not
depend on the initial choice of the countable dense set $D$
and we will abuse notation writing $E^\infty(X,d_X)$ for
$E^\infty(X,d_X,D)$, unless this can cause confusion. Note
that the space $E^\infty(X,d_X)$ is actually isometric to
the Urysohn sphere since it realizes all finitely supported
Kat\v etov functions with values in $[0,1]$. Note also that
if $\varphi:X\to X$ is an isometry, then $\varphi$ extends
canonically to an isometry
$\varphi^\infty:E^\infty(X,d_X)\to E^\infty(X,d_X)$. This
follows in the same way as the analogous extension property
is proved for the Kat\v etov extensions \cite[Page
115]{pestov} from the (above mentioned) fact that $\varphi$
extends to $E(X,d_X)$.

Given a metric space $(X,d_X)$ and its countable dense
subset $D$, write $d_{E^\infty(X,d_X)}D^\infty(X,d_X)$ for
the family of functions on $X$ of the form $x\mapsto
d_{E^\infty(X,d_X)}(x,d)$ for $d\in D^\infty(X,d_X,D)$.
Write also $\ring^\infty(X,d_X,D)$ for the ring of functions
on $X$ generated by
$d_{E^\infty(X,d_X)}D^\infty(X,d_X)$. The next proposition
follows rather immediately from the construction.

\begin{proposition}\label{huhu}
  Let $(X,d_X)$ be a separable metric space of diameter
  bounded by 1 and let $D\subseteq X$ be a countable dense
  set.  Then $\ring^\infty(X,d_X,D)$ is contained in the
  linear span of $d_{E^\infty(X,d_X)}D^\infty(X,d_X)$.
\end{proposition}
\begin{proof}
  Write $D^\infty$ for $D^\infty(X,d_X,d)$ and $dD^\infty$
  for $d_{E^\infty(X,d_X)}D^\infty(X,d_X)$. Every function
  $f$ in $\ring^\infty(X,d_X,D)$ is a polynomial in finitely
  many functions in $dD^\infty$, which in turn are the
  distance functions to finitely many points in
  $D^\infty$. Thus, there is $n\in\N$ such that $f$ belongs
  to the ring generated by the distance functions to
  $D^n(X,d_X,D)$. Then, $f$ belongs to the linear span of
  distance functions to the points in $D^{n+1}(X,d_X,D)$.
\end{proof}

Now, fix a countable dense set $D\subseteq\U_1$ and write
$D^\infty$ for $D^\infty(\U_1,d_{\U_1},D)$ and $d_E$ for the
metric on $E^\infty(\U_1,d_{\U_1})$. Consider the convex
compact set $S(E^\infty(\U_1,d_{\U_1}),d_E
D^\infty(\U_1,d_{\U_1},D))$ and write $S(\U_1,d_{\U_1})$ for
the closed convex hull of $\U_1$ in
$S(E^\infty(\U_1,d_{\U_1}),d_E
D^\infty(\U_1,d_{\U_1},D))$. Note that $S(\U_1,d_{\U_1})$ is
equal to $S(\U_1,d_{U_1},d_ED^\infty)$.  Note also that by
Proposition \ref{usaturation} and the fact that
$E^\infty(\U_1,d_{\U_1})$ is isometric to the Urysohn
sphere, the family $d_ED^\infty$ is saturated.

Recall (Proposition \ref{zsets}) that given a subspace
$(X,d_X)$ of $\U_1$ we write $Z(X,d_X)$ for an isometric
copy of $(X,d_X)$ appropriately embedded into $\U_1$.

\begin{definition}
  For a subspace $(X,d_X)$ of the Urysohn sphere
  $(\U_1,d_{\U_1})$, write $S_Z(X,d_X)$ for
  $S(Z(X,d_X),d_ED^\infty)$.
\end{definition}

Note that by Proposition \ref{indep3} and the remarks
proceeding Proposition \ref{huhu}, up to affine
homeomorphism, the above definition does not depend on the
choice of the dense countable set $D\subseteq\U_1$.

Similarly as with Proposition \ref{usaturation}, the
following result is stated for the Urysohn sphere but holds
true for the Urysohn space as well.

\begin{proposition}\label{usimplex}
  The convex compact set $S(\U_1,d_{\U_1})$ is a simplex.
\end{proposition}
\begin{proof}
  Write $S$ for $S(\U_1,d_{\U_1})$ and $d_S$ for the metric
  on $S$ (induced from the Hilbert cube). Recall that we
  have fixed a countable dense set $D\subseteq \U_1$ and
  write $d_E$ for the metric on $E^\infty(\U_1,d_{\U_1},D)$.
 
  Suppose $\mu$ and $\nu$ are two distinct Borel probability
  measures in $P(S)$ which are supported on $\ext(S)$. Since
  as $\U_1\subseteq \ext(S)$ is dense by Corollary
  \ref{density}, we can pick two sequences of Borel
  probability measures $\mu_n\to\mu$ and $\nu_n\to\nu$ such
  that $\mu_n$ and $\nu_n$ are finitely supported on $\U_1$.

  Since $\mu$ and $\nu$ are distinct as elements of $P(S)$,
  and the coordinate functions (from $[0,1]^\N$) separate
  points in $S$, by the Stone--Weierstrass theorem, there is
  a function $f\in C(S)$ which belongs to the ring generated
  by the coordinate functions and is such that $\int
  fd\mu\not=\int fd\nu$. Note that the restrictions of the
  coordinate functions to $\U_1$ are equal to the
  $d_E$-distance functions to the points in $D^\infty$. By
  Proposition \ref{huhu} and the fact that $\U_1$ is dense
  in $\ext(S)$, we can assume that the restriction of $f$ to
  $\U_1$ is actually equal to the $d_E$-distance function
  to, say, $z\in D^\infty$. Thus, we assume that the
  function $f$ on $\ext(S)$ is just one of the coordinate
  functions. Say it is the $k$-th coordinate, i.e. $z$ is
  the $k$-th element of $D^\infty$.

  Write $x$ for the barycenter of $\mu$ and $y$ for the
  barycenter of $\nu$ and let $x_n$ be the barycenter of
  $\mu_n$ and $y_n$ of $\nu_n$. Note that $x_n\to x$ and
  $y_n\to y$ \cite[Chapter 15, Proposition
  2.2]{handbook}. Now, since $\mu_n$ and $\nu_n$ are
  supported on $\U_1$, similarly as in Propoxition
  \ref{bauer}, we get that $x_n(k)=\int fd\mu_n$ and
  $y_n(k)=\int fd\nu_n$.  But $\int fd\mu_n\to\int fd\mu$
  and $\int fd\nu_n\to\int fd\nu$. Thus, $x(k)=\int fd\mu$
  and $y(k)=\int fd\nu$ and hence $x$ and $y$ are distinct,
  as needed. This ends the proof.
\end{proof}

\begin{corollary}\label{simplex}
  For every closed subspace $(X,d_X)$ of $\U_1$ the set
  $S_Z(X,d_X)$ is a simplex and $Z(X,d_X)$ is a dense subset
  of $\ext(S_Z(X,d_X))$.
\end{corollary}
\begin{proof}
  For simplicity identify $(X,d_X)$ with $Z(X,d_X)$. Note
  that $S_Z(X,d_X)$ is equal to the closed convex hull of
  $X$ in $S(\U_1,d_{\U_1})$ and $X$ is contained in
  $\ext(S(\U_1,d_{\U_1}))$. Thus, $S_Z(X,d_X)$ is a face of
  $S(\U_1,d_{\U_1})$. By Proposition \ref{usimplex}, this
  implies \cite[Chapter 15, Corollary 3.3]{handbook} that
  $S_Z(X,d_X)$ is a simplex.  Clearly, $X$ is contained in
  $\ext(S_Z(X,d_X))$. The fact that $X$ is dense in
  $\ext(S_Z(X,d_X))$ follows directly from Corollary
  \ref{density}.
\end{proof}

Finally, we need to see that $S$-extensions of metric spaces
are invariant under the isometry of the metric spaces.

\begin{proposition}\label{firstinvariance}
  Suppose $(X,d_X)$ and $(Y,d_Y)$ subspaces of $\U_1$. If
  $(X,d_X)$ and $(Y,d_Y)$ are isometric, then the simplices
  $S_Z(X,d_X)$ and $S_Z(Y,d_Y)$ are affinely homeomorphic
  via a map that extends the isometry of $Z(X,d_X)$ and
  $Z(Y,d_Y)$.
\end{proposition}
\begin{proof}
  For simplicity again identify $(X,d_X)$ with $Z(X,d_X)$
  and $(Y,d_Y)$ with $Z(Y,d_Y)$. Let $\varphi:\U_1\to\U_1$
  be an isometry such that $\varphi''X=Y$. Let $D\subseteq
  \U_1$ be the fixed countable dense set and let
  $E=\varphi''D$. Write
  $\varphi^\infty:E^\infty(\U_1,d_{\U_1})\to
  E^\infty(\U_1,d_{\U_1})$ for the extension of $\varphi$
  and note that $(\varphi^\infty)''D^\infty=E^\infty$. Write
  $d$ for the metric on $E^\infty(\U_1,d_{\U_1})$. Note that
  since $\varphi^\infty$ is an isometry, the sets
  $S(E^\infty(\U_1,d_{\U_1}),D^\infty, d D^\infty)$ and
  $S(E^\infty(\U_1,d_{\U_1}),E^\infty,d E^\infty)$ are equal
  and we get the following diagram:
  \begin{center}
    \begin{tikzcd}
      S(E^\infty(\U_1,d_{\U_1}),D^\infty, d D^\infty)\arrow{rr}{\id}&   &S(E^\infty(\U_1,d_{\U_1}),E^\infty,d E^\infty)\\
      \U_1 \arrow{u}{i_{d D^\infty}}\arrow{rr}{\varphi} & &
      \U_1\arrow{u}[right]{i_{d E^\infty}}
    \end{tikzcd}
  \end{center}
  Composing this with the diagram in Proposition
  \ref{indep3} (where we take the restrictions of the maps
  $i_{d D^\infty}$ and $i_{d E^\infty}$ to $\U_1$), we get
  \begin{center}
    \begin{tikzcd}
      S(E^\infty(\U_1,d_{\U_1}),D^\infty, d D^\infty)
      \arrow{rr}{\tau_{D^\infty}^{E^\infty}}&   &S(E^\infty(\U_1,d_{\U_1}),D^\infty, d D^\infty) \\
      \U_1 \arrow{u}{i_{d D^\infty}}\arrow{rr}{\varphi} & &
      \U_1\arrow{u}[right]{i_{d D^\infty}}
    \end{tikzcd}
  \end{center}
  and since all the maps are affine and continuous, this
  immediately implies that $\tau_{D^\infty}^{E^\infty}$ maps
  $\conv(i_{d D^\infty}X)=S_Z(X,d_X)$ to $\conv(i_{d
    D^\infty}Y)=S_Z(Y,d_Y)$.
\end{proof}

The simplex $S(\U_1,d_{\U_1})$ plays now the role of a
universal homogeneous Choquet simplex. However, it does not
seem to be affinely homeomorphic to the Poulsen simplex. On
the other hand, that there is another convex compact set,
which seems to be related to the Poulsen simplex. It is the
set $S(\U_1,d_{\U_1},D,d_{\U_1}D)$. Write
$S'(\U_1,d_{\U_1})$ for $S(\U_1,d_{\U_1},D,d_{\U_1}D)$ (for
a countable dense set $D\subseteq\U_1$). Again, up to affine
homeomorphism, it does not depend on the choice of the
countable set $D$. Now, the map from $S(\U_1,d_{\U_1})$ to
$S'(\U_1,d_{\U_1})$, which forgets about the coordinates
corresponding to $D^\infty\setminus D$ is clearly affine and
continuous and hence $S(\U_1,d_{\U_1})$ can be treated as an
unfolded version of $S'(\U_1,d_{\U_1})$. We do not know if
$S'(\U_1,d_{\U_1})$ is a simplex but it seems to be more
closely related to the Pousen simplex than
$S(\U_1,d_{\U_1})$.

\begin{proposition}\label{poulsen}
  $S'(\U_1,d_{\U_1})$ has a dense set of extreme points.
\end{proposition}

The proof follows from the following simple claim.

\begin{claim}\label{katetovconv}
  If $g$ and $h$ are Kat\v etov functions and
  $\alpha\in[0,1]$, then $\alpha f+(1-\alpha)g$ is also a
  Kat\v etov function.
\end{claim}
\begin{proof}
  This is an elementary computation. The Kat\v etov
  inequalities for $\alpha f+(1-\alpha)g$ follow directly
  from the Kat\v etov inequalities for $f$ and $g$.
\end{proof}

\begin{proof}[Proof of Proposition \ref{poulsen}]
  By the fact that $\U_1\subseteq\ext(S'(\U_1,d_{\U_1}))$,
  it is enough to check that $\U_1$ is dense in
  $S'(\U_1,d_{\U_1})$. Let $s\in S'(\U_1,d_{\U_1})$ and
  $\varepsilon>0$. Find $n$ such that
  $2^{-(n+1)}<\varepsilon\slash 2$ and $s$ is
  $\varepsilon$-close to
  $S^n_n(\U_1,d_{\U_1},D,d_{\U_1}D)$. Write
  $D=(d_1,d_2,\ldots)$ and $\U_1\restriction n$ for the
  image of $\U_1\subseteq S(\U_1,d_{\U_1},D,d_{\U_1}D)$
  under the projection map from $[0,1]^\N$ to
  $[0,1]^n$. Note that it suffices to show that
  $S^n_n(\U_1,d_{\U_1},D,d_{\U_1}D)$ is contained in
  $\U_1\restriction n$. Clearly, the vertices of the simplex
  $S^n_n(\U_1,d_{\U_1},D,d_{\U_1}D)$ belong to
  $\U_1\restriction n$. For each $i\leq n$ let
  $f_i:\{d_1,\ldots,d_n\}\to[0,1]$ be defined
  $f_i(z)=d(d_i,z)$. Now, Claim \ref{katetovconv} implies
  that for every $\alpha_1,\ldots,\alpha_n\in[0,1]$ with
  $\sum_i\alpha_i=1$ the function $\sum_i \alpha_i f_i$ is a
  Kat\v etov function with values in $[0,1]$, thus realized
  in $\U_1$. This implies that the set of points
  $(d(x,d_1),\ldots,d(x,d_n))\in[0,1]^n$ for $x\in\U_1$ is
  convex. But the latter set is equal to $\U_1\restriction
  n$. Thus, since the the vertices of
  $S_n^n(\U_1,d_{\U_1},D,d_{\U_1}D)$ belong to
  $\U_1\restriction n$ and $\U_1\restriction n$ is convex,
  we have that $S_n^n(\U_1,d_{\U_1},D,d_{\U_1}D)$ is
  contained in $\U_1\restriction n$, as needed. This ends
  the proof.
\end{proof}

\section{Iterated cone construction}\label{sec:cones}

Given a simplex $S$ and a point $s\in S$ we define the
\textit{cone of $S$ over $s$} as follows. Consider
$Y=S\times[0,1]$ and let
$$\cone(S,s)=\mathrm{conv}((S\times\{0\})\cup(s,1)),$$
where the convex hull is taken in $Y$ (recall our convention
on metrics in Section \ref{sec:convex} and note that if $S$
is a subset of the Hilbert cube, then the cone is embedded
in the Hilbert cube as well).

The point $(s,1)$ is called the \textit{cone point} of the
cone and we will identify $S$ with
$S\times\{0\}\subseteq\cone(S,s)$. We also denote the cone
point as $c(s)$. The cone admits a natural affine continuous
map $\pi:\cone(S,s)\rightarrow S$ which is just the
projection map in $Y$. It maps the cone point $c(s)$ to the
point $s$.

We state now a couple of basic facts on the structure of
cones.

\begin{lemma}\label{face}
  If $S$ is a simplex and $s\in S$, then $\cone(S,s)$ is a
  simplex, $S$ is a face of $\cone(S,s)$ and
  $\ext(\cone(S,s))=\ext(S)\cup\{c(s)\}$.
\end{lemma}
\begin{proof}
  Note that if two points $z_1,z_2\in S\times[0,1]$ have an
  affine combination that lies in $S$, then both $z_1$ and
  $z_2$ must belong to $S$. This implies that $S$ is a face
  of the convex compact set $\cone(S,s)$ and that
  $\ext(S)\subseteq\ext(\cone(S,s))$. It is clear that
  $c(s)$ is an extreme point of $\cone(S,s)$, so indeed
  $\ext(\cone(S,s))=\ext(S)\cup\{c(s)\}$. To see that
  $\cone(S,s)$ is a simplex note that if $\mu$ is a measure
  concertated on $\ext(S)\cup\{c(s)\}$ such that $\int
  fd\mu=0$ for every affine continuous function on
  $\cone(S,s)$, then $\mu(\{c(s)\})=0$ and hence $\mu$ must
  be concentrated on $\ext(S)$.
\end{proof}

\begin{lemma}\label{doublecone}
  Given a simplex $S$ and two points $s_1,s_2\in S$ we
  have $$\cone(\cone(S,s_1),s_2)\simeq\cone(\cone(S,s_2),s_1).$$
\end{lemma}
\begin{proof}
  Both simplices are affinely homeomorphic to the closed
  convex hull of $S\times\{(0,0)\}$, $(s_1,0,1)$ and
  $(s_2,1,0)$ in $S\times[0,1]^2$. To see the affine
  homeomorphism of $\cone(\cone(S,s_1),s_2)$ and
  $\cone(\cone(S,s_2),s_1)$ directly, write $c(s_1)$ for the
  cone point of $\cone(S,s_1)$, $c'(s_2)$ for the cone point
  of $\cone(\cone(S,s_1),s_2)$, $c(s_2)$ for the cone point
  of $\cone(S,s_2)$ and $c'(s_1)$ for the cone point of
  $\cone(\cone(S,s_2),s_1))$. The affine homeomorphism then
  maps $(1-\beta)((1-\alpha) x +\alpha c(s_1))+\beta
  c'(s_2)$ to $(1-\gamma)((1-\delta)x+\delta c(s_2))+\gamma
  c'(s_1)$, where $\gamma=(1-\beta)\alpha$ and
  $\delta=\beta\slash(1-\alpha(1-\beta))$.
\end{proof}

Given the above lemma, we use the notation
$\cone(S,s_1,\ldots,s_n)$ to denote an iterated cone of the
form $\cone(\ldots\cone(S,s_1)\ldots,s_n)$ (or with any
other permutation). We also call a cone of the form
$\cone(S,s_1,s_2)$ a \textit{double cone}. Given a subset
$\{i_1,\ldots,i_k\}\subseteq\{1,\ldots,n\}$, the simplex
$\cone(S,s_{i_1},\ldots,s_{i_k})$ is a subset of
$\cone(S,s_1,\ldots,s_n)$ in the same way $S$ is a subset of
$\cone(S,s)$. The simplex $\cone(S,s_{i_1},\ldots,s_{i_k})$
is then a face of $\cone(S,s_1,\ldots,s_n)$. We call a face
of the form $\cone(S,s_{i_1},\ldots,s_{i_k})$ in
$\cone(S,s_1,\ldots,s_n)$ a \textit{subcone}.

\begin{lemma}\label{conecomplement}
  Given a simplex $S$ and $s_1,\ldots,s_n\in S$, if
  $\{i_1,\ldots,i_k\}\subseteq\{1,\ldots,n\}$ and
  $\{j_1,\ldots,j_{n-k}\}=\{1,\ldots,n\}\setminus\{i_1,\ldots,i_k\}$,
  then $$\cone(S,s_1,\ldots,s_n)\simeq\cone(\cone(S,s_{i_1},\ldots,s_{i_k}),s_{j_1},\ldots,s_{j_{n-k}}).$$
\end{lemma}
\begin{proof}
  This follows directly from the definitions and Lemma
  \ref{doublecone}.
\end{proof}
\begin{definition}
  Given a simplex $S$ and its countable (enumerated) subset
  $D=(d_n:n\in\N)$, define the \textit{iterated cone over
    $D$} as follows. Let $S_0=S$ and for each $n\in\N$ let
  $S_n=\cone(S_{n-1},d_n)$ and let $\pi_n:S_n\rightarrow
  S_{n-1}$ be the projection
  map. Define $$\cone(S,D)=\varprojlim(S_n,\pi_n).$$
\end{definition}

Note that if $S$ is a subset of the Hilbert cube, then
$\cone(S,D)$ can be naturally embedded into the Hilbert cube
as well. Note also that the inverse system in the above
definition is increasing. Thus, we treat $S_n$ as a face of
$\cone(S,D)$, for each $n\in\N$. Given that, all cone points
of the simplices $S_n$ belong to $\cone(S,D)$ and we refer
to them as to the \textit{cone points of the iterated
  cone}. In principle, the iterated cone over $D$ may depend
on the ordering of the set $D$. As we will see, this does
not happen.

\begin{lemma}\label{extension}
  Let $S$ be a simplex, $\varepsilon>0$ and
  $\varphi:S\rightarrow S$ be an affine homeomorphism. Given
  two points $s_1,s_2\in S$ if
  $d_S(\varphi(s_1),s_2)<\varepsilon$, then there is an
  affine homeomorphism
  $\varphi':\cone(S,s_1)\rightarrow\cone(S,s_2)$ extending
  $\varphi$ such
  that $$d_S(\varphi(\pi_1(s)),\pi_2(\varphi'(x)))<\varepsilon$$
  for every $s\in\cone(S,s_1)$, where $\pi_1:\cone(S,s_1)\to
  S$ and $\pi_2:\cone(S,s_2)\to S$ are the projection maps.
\end{lemma}
\begin{proof}
  The map $\varphi'$ is the affine map whose restriction to
  $S$ is $\varphi$ and which maps $c(s_1)$ to $c(s_2)$. To
  see that
  $d_S(\pi_2(\varphi'(s)),\varphi(\pi_1(s)))<\varepsilon$
  for each $s\in\cone(S,s_1)$, note that
  $\pi_2\circ\varphi'$ and $\varphi\circ\pi_1$ are two
  affine continuous maps such that
  $d_S(\pi_2(\varphi'(e)),\varphi(\pi_1(e)))<\varepsilon$
  holds for every extreme point of $\cone(S,s_1)$. Thus, by
  Proposition \ref{affproximity},
  $d_S(\pi_2(\varphi'(s)),\varphi(\pi_1(s)))<\varepsilon$ is
  true for every $s\in\cone(S,s_1)$.
\end{proof}

\begin{proposition}\label{correctness}
  Given a simplex $S$ and an affine homeomorphism
  $\varphi:S\to S$, if $D=(d_n:n\in\N)$ and $E=(e_n:n\in\N)$
  are two subsets of $S$ such that $\cl(D)=\cl(\varphi''E)$,
  then there is an extension $\varphi'$ of $\varphi$ to an
  affine homeomorphism $\varphi':\cone(S,D)\to\cone(S,E).$
\end{proposition}
\begin{proof}
  Write $X=\cone(S,D)$ and $Y=\cone(S,E)$ and
  define $X_n$ and $Y_n$ as $X_0=Y_0=S$ and
  $X_{n+1}=\cone(X_n,d_n)$ and $Y_{n+1}=\cone(Y_n,e_n)$ so
  that $X=\varprojlim X_n$ and $Y=\varprojlim Y_n$. We
  define inductively two increasing sequences of natural
  numbers $n_i$ and $m_i$ and affine continuous maps
  $\varphi_i:X_{n_i}\rightarrow Y_{m_i}$ and
  $\psi_i:Y_{m_i}\rightarrow X_{n_{i+1}}$ such that
  $\varphi_0=\varphi$ and
  \begin{itemize}
  \item $(\varphi_n,\psi_n:n\in\N)$ is an approximate
    intertwining.
  \item $\rng(\varphi_n)$ and $\rng(\psi_n)$ is a subcone of
    $X$ or $Y$, for each $n\in\N$.
  \end{itemize}
  The theorem will then follow by Proposition
  \ref{approxinter}.

  Let $\varphi_0=\psi_0^{-1}=\varphi$. Suppose
  $\varphi_{i-1},\psi_{i-1}$ are defined. We need to define
  $\varphi_i:X_{n_i}\rightarrow Y_{m_{i+1}}$. Let
  $\varphi_i'=\psi_i^{-1}$ and
  $X_{n_i}'=\rng(\psi_i)=\dom(\varphi_i')$. By our
  assumption, $X_{n_i}'$ is a subcube, so there are
  $p_1,\ldots,p_k\in\N$ such that
  $X_{n_i}'=\cone(S,d_{p_1},\ldots,d_{p_k})$. Find
  $n_{i+1}>n_i$ such that $p_j<n_{i+1}$ for each $j\leq
  k$. Let
  $\{q_1,\ldots,q_l\}=\{1,\ldots,n_{i+1}\}\setminus\{p_1,\ldots,p_k\}$. By
  Lemma \ref{conecomplement}, $X_{n_i}$ is affinely
  homeomorphic to
  $\cone(X_{n_i}',d_{q_1},\ldots,d_{q_l})$. We inductively
  find increasing numbers $m_i^j\in\N$ and maps
  $\varphi_i^j:\cone(X_{n_i}',d_{q_1},\ldots,d_{q_j})\rightarrow
  Y_{m_i^j}$ so that $\varphi_i^0=\varphi_i'$,
  $\varphi_i^j\subseteq\varphi_i^{j+1}$ and for each $j$ we
  have
  \begin{equation}\label{eq2}
    |d_{Y_{m_i}}(\varphi_i(x\restriction X_{n_i}),\varphi_i^j(x)\restriction Y_{m_i})|<2^{i+j}
  \end{equation}
  for each $x\in\cone(X_{n_i},d_{q_1},\ldots,d_{q_j})$.
  Given $\varphi_i^j$, find $m_i^{j+1}$ such that
  $$d_{S}(\varphi(d_{q_{j+1}}),e_{m_i^{j+1}})=d_{Y_{m_i^j}}(\varphi_i^j(d_{q_{j+1}}),e_{m_i^{j+1}})<2^{i+j}$$
  and use Lemma \ref{extension} to extend $\varphi_i^j$ to
  $\varphi_i^{j+1}$ so so that (\ref{eq2}) holds. At the
  end, put $\varphi_i=\varphi_i^l$ and $m_i=m_i^l$. The map
  $\psi_{i+1}$ is defined analogously.

  The condition (\ref{eq2}) and an analogous inequality for
  $\psi_i$ witness that $\varphi_i$ and $\psi_i$ form an
  approximate intertwining.
\end{proof}

Applying Proposition \ref{correctness} to $\varphi=\id$ we
get that that the definition of $\cone(S,D)$ indeed does not
depend on the ordering of $D$.

\begin{lemma}\label{infcorrectness}
  Given a simplex $S$, with a countable subset $D$ and $d\in
  S\setminus D$ we have $$\cone(S,\{d\}\cup
  D)\simeq\cone(\cone(S,D),d).$$
\end{lemma}
\begin{proof}
  Note first that if $(X_n,\pi_n:n\in\N)$ is an increasing
  inverse system of simplices and $x\in X_0$, then
  $\cone(\varprojlim X_n,x)\simeq\varprojlim\cone(X_n,x)$,
  where the latter inverse system has the natural projection
  maps $\pi_n':\cone(X_n,x)\to\cone(X_{n-1},x)$ that extend
  $\pi_n$ and map the cone point to the cone point.

  Write $D=\{d_1,d_2,\ldots\}$ and
  $X_n=\cone(S,d_1,\ldots,d_n)$. By the remark above,
  $\cone(\cone(S,D),d)\simeq\varprojlim\cone(X_n,d)$ and by
  Lemma \ref{doublecone} we have
  $\cone(X_n,d)\simeq\cone(S,d,d_1,\ldots,d_n)$. Thus,
  $\cone(\cone(S,D),d)$ is affinely homeomorphic to
  $\varprojlim(\cone(S,d,d_1,\ldots,d_n))$, which is equal
  to $\cone(S,\{d\}\cup D)$.
\end{proof}

Now we look at the extreme point of the iterated cones.

\begin{lemma}\label{suspension}
  Suppose $S$ is a simplex and $s_1,s_2\in S$. If
  $y\in\cone(S,s_1)$ is such that $$y=\alpha
  x+(1-\alpha)s_1$$ for some $x\in S$ and $0<\alpha<1$, then
  for every $y'\in\cone(\cone(S,s_1),s_2)$ with $\pi(y')=y$
  there exists a unique $x'\in\cone(S,s_2)$ such
  that $$y'=\alpha x'+(1-\alpha)s_1,$$ where
  $\pi:\cone(\cone(S,s_1),s_2)\to\cone(S,s_1)$ is the
  projection map.
\end{lemma}
\begin{proof}
  Uniqueness of $x'$ is immediate. We only need to find $x'$
  in the the cone. Put $y''=\alpha s_2 +(1-\alpha)s_1$ and
  note that $\pi(y'')=y$. Since $\cone(\cone(S,s_1),s_2)$ is
  convexely generated by $\cone(S,s_1)$ and $y''$, there is
  $\beta\in[0,1]$ such that $y'=\beta y+(1-\beta)y''$. Put
  $x'=\beta x+(1-\beta)s_2$. We claim that $x'$ is as
  needed, i.e. that $\alpha x'+(1-\alpha)s_1=y'$. And indeed,
 \begin{eqnarray*}
   y'=\beta y+(1-\beta)y''
   =\beta(\alpha
   x+(1-\alpha)s_1)+(1-\beta)(\alpha s_2 +(1-\alpha)s_1)\\
   =\alpha\beta x +\beta(1-\alpha)s_1+(1-\beta)\alpha s_2 +(1-\beta)(1-\alpha)s_1\\
   =\alpha\beta x+(1-\alpha)s_1+(1-\beta)\alpha s_2\\
   =\alpha(\beta x+(1-\beta)s_2)+(1-\alpha)s_1=\alpha x'+(1-\alpha)s_1
  \end{eqnarray*}
\end{proof}

\begin{lemma}\label{extcone}
  Given a simplex $S$ and its countable subset $D$ let
  $\{v_n:n\in\N\}$ be the set of cone points of
  $\cone(S,D)$. Then $$\ext(\cone(S,D))=\ext(S)\cup
  \{v_n:n\in\N\}.$$
\end{lemma}
\begin{proof}
  By Lemma \ref{face}, $S$ is a face of $\cone(S,D)$, so
  $\ext(S)\subseteq\ext(\cone(S,D))$. Also, for each $n$ the
  simplex $\cone(S,d_0,\ldots,d_n)$ is a face of
  $\cone(S,D)$, which shows that
  $v_n\in\ext(\cone(S,D))$. We need to show that there are
  no more extreme points in $\cone(S,D)$. Let
  $y\in\cone(S,D)$ be such that $y\notin S$ and $y\not=v_n$
  for each $n$. Pick $n$ such that $y\restriction n\notin
  S$. This means that there exists $\lambda$ with
  $0<\lambda<1$ and $s\in S$ such that $$y\restriction
  n=\lambda s + (1-\lambda)v_n.$$ By Lemma \ref{suspension},
  for each $m>n$ there exists
  $s_m\in\cone(S,d_{n+1},\ldots,d_m)$ such
  that $$y\restriction m=\lambda s_m+(1-\lambda)v_n.$$
  Uniqueness of $s_m$ implies that $s_{m_2}\restriction
  m_1=s_{m_1}$ for each $m_2>m_1>n$. Let
  $x\in\cone(S,\{d_{n+1},d_{n+2},\ldots\})$ be such that
  $x\restriction k=s_k$ for each $k>n$. By Lemma
  \ref{infcorrectness}, we have
  $$\cone(S,D)=\cone(\cone(S,\{d_{n+1},d_{n+2},\ldots\}),d_0,\ldots,d_n)$$
  and in the latter simplex we have $y=\lambda x +
  (1-\lambda)v_n$, which shows that $y$ is not an extreme
  point of $\cone(S,D)$.
\end{proof}

\begin{lemma}\label{isolated}
  Given a simplex $S$ and its countable subset $D$, let
  $\{v_n:n<\omega\}$ be the set of cone points of
  $\cone(S,D)$. Then each $v_n$ is an isolated point of
  $\ext(\cone(S,D))$.
\end{lemma}
\begin{proof}
  Write $D=\{d_1,d_2,\ldots\}$ so that $v_n$ is the cone
  point over $d_n$, for each $n\in\N$. We need to show that
  $v_n$ is isolated in $\ext(S)\cup\{v_i:i\not=n\}$. Write
  $\pi^\infty_n:\cone(S,D)\to\cone(S,d_1,\ldots,d_n)$ for
  the projection map. Clearly, $v_n=\pi^\infty_n(v_n)$ has
  positive distance from $S\cup\{v_i:i<n\}$ in
  $\cone(S,d_1,\ldots,d_n)$, so let $U$ be an open
  neighborhood of $v_n$ in $\cone(S,d_1,\ldots,d_n)$ that is
  disjoint from $S$ and $\{v_i:i<n\}$. Then
  $(\pi^\infty_n)^{-1}(U)$ is an open neighborhood of $v_n$
  that is disjoint from $\ext(S)$ and does not contain any
  $v_i$ with $i\not=n$.
\end{proof}

\section{The blow-up construction}\label{sec:blowup}

Recall that given a simplex $S$ and two points $x_1,x_2\in
S$ the double cone $\cone(S,x_1,x_2)$ is the simplex
$\cone(\cone(S,x_1),x_2)$. Now, given a simplex $S$ and its
subset $X\subseteq S$ together with a metric $d_X$ on $X$ we
will define another, bigger, simplex, which will be used to
encode the metric in its affine structure.

Given a countable dense subset $D$ of a metric space
$(X,d_X)$, enumerate as $(p_n=(x_n,y_n,U_n):n\in\N)$, with
infinite repetitions, all triples $(x,y,U)$ with $x,y\in D$
distinct and $U\subseteq[0,1]$ basic open with $d_X(x,y)\in
U$. Call the sequence $(p_n:n\in\N)$ the \textit{metric
  scheme} of $(X,d_X,D)$ and denote it by $\Sch(X,d_X,D)$.
For every $n\in\N$ and $p_n=(x_n,y_n,U)$ in the metric
scheme, write $p_n(1)$ for $x_n$ and $p_n(2)$ for
$y_n$. Write also $p_n(D)$ for the $n$-th element of
$\Sch(X,d_X,D)$.

Define an increasing inverse system of simplices as
follows. Let $B_0=S$ and
$B_n=\cone(B_{n-1},p_n(1),p_n(2))$. Define the
\textit{blow-up} of $S$ with respect to $(X,d_X)$ (slightly
abusing the notation), denoted by $B(S,X,d_X)$ as
$\varprojlim B_n$ and write $c_1(p_n)$ for the cone point
over $p_n(1)$ in $B(S,X,d_X)$ and $c_2(p_n)$ for the cone
point over $p_n(2)$ in $B(S,X,d_X)$.

\begin{definition}
  Given a closed subspace $(X,d_X)$ of $\U_1$ define
  $B(X,d_X)$ as $B(S_Z(X,d_X),X,d_X)$.
\end{definition}

In principle, the blow-up construction depends only on the
enumeration (with repetitions) of pairs of points in $D$. We
keep, however, the sequence of $p_n$'s for reference to the
further construction. Note that each pair of distinct points
in $D$ appears infinitely often in the sequence. Analogous
arguments as in Section \ref{sec:cones}, show that the
blow-up depends neither on the ordering $p_n$'s nor on the
choice of the dense set.

\begin{proposition}\label{blowupcorrect}
  Suppose $S$ and $T$ are simplices, $X\subseteq S$ and
  $Y\subseteq T$ are their subsets, $d_X$ and $d_Y$ are
  metrics on $X$ and $Y$, respectively and $D\subseteq X$
  and $E\subseteq Y$ are countable dense sets. Given an
  affine homeomorphism $\varphi: S\to T$ such that
  $\varphi''X=Y$ and $\varphi\restriction X$ is an isometry
  of $(X,d_X)$ and $(Y,d_Y)$, there is an affine
  homeomorphism $\bar\varphi:B(S,X,d_X)\to B(T,Y,d_Y)$ such
  that $\bar\varphi$ extends $\varphi$ and
  \begin{itemize}
  \item[(\textit{a})] for every $p\in\Sch(X,d_X,D)$ there
    exists $q\in\Sch(Y,d_Y,E)$ such that if $p=(x_1,x_2,U)$,
    then $q=(y_1,y_2,U)$ and $\bar\varphi$ maps $c_1(p)$ to
    $c_1(q)$ and $c_2(p)$ to $c_2(q)$,
  \item[(\textit{b})] for every $q\in\Sch(Y,d_Y,E)$ there
    exists $p\in\Sch(X,d_X,D)$ such that if $q=(y_1,y_2,V)$,
    then $p=(x_1,x_2,V)$ and $\bar\varphi^{-1}$ maps
    $c_1(q)$ to $c_1(p)$ and $c_2(q)$ to $c_2(p)$.
  \end{itemize}
\end{proposition}
\begin{proof}
  The proof is essentially the same as that of Proposition
  \ref{correctness} and we only sketch it. Write $B_n$ for
  the simplices in the inverse system of $B(S,X,d_X)$ and
  $C_n$ for the simplices in the invserse system of
  $B(T,Y,d_Y)$. By induction on $i\in\N$, construct an
  approximate intertwining $\varphi_i: B_{n_i}\to C_{m_i}$
  and $\psi_i:C_{m_i}\to B_{n_{i+1}}$ so that
  $\varphi_0=\varphi$. At the inductive construction, when
  extending a map from $B_n$ to $B_{n+1}$, make sure that if
  $p_n(D)=(x_1,x_2,U)$, then for some $m\in\N$ with
  $p_m(E)=(y_1,y_2,U)$, the point $c_1(p_n(E))$ is mapped to
  $c_1(p_m(E))$, $c_2(p_n(E))$ is mapped to $c_2(p_m(E))$ so
  that
  \begin{equation}\label{ineq5}
    d_T(\varphi(x_1),y_1),d_T(\varphi(x_2),y_2)<2^{-m}
  \end{equation}
  and $d_Y(y_1,y_2)\in U$. Analogous conditions apply when
  we extend a map from $C_m$ to $C_{m+1}$. The fact that the
  above is possible follows at once from the fact that
  $\varphi$ was an isometry and the tripples in the schemes
  are enumerated with infinite repetitions.

  Once the construction is finished, (\ref{ineq5}) and the
  symmetric condition involving $d_S$ imply that the
  sequence forms an approximate intertwining. Then, an
  application of Proposition \ref{twisted} gives an affine
  homeomorphism $\bar\varphi:\varprojlim B_i\to\varprojlim
  C_i$. The conditions (a) and (b) follow directly from the
  construction.
\end{proof}

\section{Coding a metric into the affine structure of a
  simpliex}\label{sec:final}

Given a metric space $(X,d_X)$ of diameter bounded by 1,
which we assume is a subspace of the Urysohn sphere $\U_1$,
and a dense countable subset $D$ of $X$, let
$(p_n=((x_n,y_n,U_n):n\in\N)$ be the associated metric
scheme.  Consider the blow-up $B(X,d_X)$ and let
$M(X,d_X,D)$ be the subset of $B(X,d_X)$ consisting of the
points $qc_1(p_n) + (1-q)c_2(p_n)$ for $q\in U_n\cap\Q$.

Use the Kuratowski--Ryll-Nardzewski theorem \cite[Theorem
12.13]{kechris}, to find a countable dense subset $D(X)$ for
every (nonempty) closed subset $X$ of $\U_1$ so that the map
$D:X\mapsto D(X)$ is Borel. Define now $\Phi(X,d_X)$ as
$\cone(B(X,d_X),M(X,d_X,D(X))$. It is easy to check that
$\Phi$ is a Borel map from the space of closed (nonempty)
subsets of $\U_1$ to the space of separable Choquet
simplices (cf. Section \ref{sec:convex} for the Borel
structure on the space of simplices). In fact, the map
$(X,D)\mapsto\cone(B(X,d_X),M(X,d_X,D))$, which maps a pair
of a Polish space and its dense countable subset to a
simplex, is continuous.

\begin{proposition}\label{homo}
  Suppose $(X,d_X)$ and $(Y,d_Y)$ are subspaces of
  $\U_1$. If $(X,d_X)$ and $(Y,d_Y)$ are isometric, then the
  simplices $\Phi(X,d_X)$ and $\Phi(Y,d_Y)$ are affinely
  homeomorphic.
\end{proposition}
\begin{proof}
  For simplicity, identify $X$ with $Z(X)$ and $Y$ with
  $Z(Y)$ (see Proposition \ref{zsets}) and let
  $\varphi:\U_1\to\U_1$ be an isometry such that
  $\varphi''X=Y$. By Propositions \ref{firstinvariance} and
  \ref{blowupcorrect}, there is an affine homeomorphism
  $\bar\varphi: B(X,d_X)\to B(Y,d_Y)$ which extends
  $\varphi$ and such that
  \begin{itemize}
  \item for each $p=(x_1,x_2,U)$ in the scheme of $(X,d_X)$
    there is $q=(y_1,y_2,V)$ in the scheme of $(Y,d_Y)$ such
    that $U=V$ and $\bar\varphi$ maps the cone points
    $c_1(p)$ to $c_1(q)$ and $c_2(p)$ to $c_2(q)$,
  \item for each $q=(y_1,y_2,V)$ in the scheme of $(Y,d_Y)$
    there is $p=(x_1,x_2,U)$ in the scheme of $(X,d_X)$ that
    $U=V$ and $\bar\varphi^{-1}$ maps $c_1(q)$ to $c_1(p)$
    and $c_2(q)$ to $c_2(p)$.
  \end{itemize}
  The above imply that
  $\bar\varphi''\cl(M(X,d_X,D(X)))=\cl(M(Y,d_Y,D(Y)))$ and
  hence, by Proposition \ref{correctness}, there is an
  affine homeomorphism
  $\bar{\bar\varphi}:\Phi(X,d_X)\to\Phi(Y,d_Y)$ that extends
  $\bar\varphi$.
\end{proof}

Proposition \ref{homo} shows that $\Phi$ is a homomorphism
from the isometry of separable metric spaces to the affine
homeomorphism of simplices. To show that $\Phi$ is a
reduction, we will restrict attention to perfect separable
metric spaces.

\begin{claim}\label{nonisolated}
  If $(X,d_X)$ is perfect, then the set of nonisolated
  extreme points of $\Phi(X,d_X)$ is equal to
  $\ext(S(X,d_X))$.
\end{claim}
\begin{proof}
  $X$ is perfect and dense in $\ext(S_Z(X,d_X))$ by
  Corollary \ref{simplex}, so all extreme points of
  $S_Z(X,d_X)$ are nonisolated in $\ext(S_Z(X,d_X))$. Since
  the simplex $S_Z(X,d_X)$ is a face of $\Phi(X,d_X)$, the
  extreme points of $S_Z(X,d_X)$ are still nonisolated extreme
  points of $\Phi(X,d_X)$. Now, Lemmas \ref{extcone} and
  \ref{isolated} (and analogous statements for the blow-up)
  imply that all the other extreme points of $\Phi(X,d_X)$
  are isolated.
\end{proof}

\begin{lemma}\label{logic}
  Suppose $(X,d_X)$ is a perfect metric subspace of
  $\U_1$. If $x,y\in X$, then $d_X(x,y)$ is the only
  $\alpha\in[0,1]$ such that
  \begin{equation}\label{eq:last}
    \begin{split}
      &\forall U\subseteq[0,1]\ \textrm{basic open
        neighborhood of }\alpha\\ &\exists
      V_x,V_y\subseteq\Phi(X,d_X)\ \textrm{open
        neighboorhoods of }x\textrm{ and }y\mbox{ in
      }\Phi(X,d_X)\\ &\forall c_1\in V_x\ \forall c_2\in
      V_y\ \textrm{ isolated extreme points }\\ &\quad\
      \big[ (\exists\lambda\in [0,1]\ \lambda
      c_1+(1-\lambda)c_2\ \\ &\quad\ \textrm{ is a limit of
        isolated
        extreme points of }\Phi(X,d_X)) \\
      &\Rightarrow( \exists\lambda\in U\ \lambda
      c_1+(1-\lambda)c_2\ \\ &\quad\ \textrm{ is a limit of
        isolated extreme points of }\Phi(X,d_X) )\big] .
    \end{split}
  \end{equation}
\end{lemma}
\begin{proof}
  Write $D$ for the countable dense subset $D(X)$ of $X$. By
  Claim \ref{nonisolated} and Lemma \ref{extcone}, the only
  isolated extreme points of $\Phi(X,d_X)$ are in the sets
  $E_1=\{c_1(p),c_2(p):p\in\Sch(X,d_X,D)\}$ and
  $E_2=\{c(z):z\in M(X,d_X,D)\}$. It follows from the
  iterated cone construction that if $z\in M(X,d_X,D)$, then
  for any isolated extreme point $e$ of $\Phi(X,d_X)$, no
  point in the set $\{\lambda
  c(z)+(1-\lambda)e:\lambda\in[0,1]\}$ is a limit of
  isolated extreme points of $\Phi(X,d_X)$. On the other
  hand, if $e_1,e_2\in E_1$ are such that for some
  $\lambda\in[0,1]$ the point $\lambda e_1+(1-\lambda)e_2$
  is a limit of isolated extreme points of $\Phi(X,d_X)$,
  then there is $p\in\Sch$ such that $e_1=c_1(p)$,
  $e_2=c_2(p)$. The latter follows from the fact that the
  intersection of $B(X,d_X)$ with the closure of the set
  $\{c(z):z\in M(X,d_X,D)\}$ is exactly the closure of
  $M(X,d_X,D)$.

  Pick $x,y\in X$. We claim that $\alpha=d_X(x,y)$ satisfies
  the condition (\ref{eq:last}). Pick any $U\subseteq[0,1]$
  basic open neighborhood of $\alpha$ and let
  $\varepsilon>0$ be such that
  $(\alpha-\varepsilon,\alpha+\varepsilon)\subseteq
  U$. Since $X$ is a topological subspace of $S(X,d_X)$,
  there are $V_x^1$ and $V_y^1$ open neighborhoods of $x$
  and $y$ in $S(X,d_X)$, respectively, such that $V_x^1\cap
  X\subseteq\ball_{(X,d_X)}(x,\varepsilon\slash 2)$ and
  $V_y^1\cap X\subseteq\ball_{(X,d_X)}(y,\varepsilon\slash
  2)$. Let $V_x$ and $V_y$ be open neighborhoods of $x$ and
  $y$ in $\Phi(X,d_X)$ such that if $c_1\in V_x$ and $c_2\in
  V_y$, then $\pi(c_1)\in V_x^1$ and $\pi(c_2)\in V_y^1$,
  where $\pi:\Phi(X,d_X)\to S(X,d_x)$ denotes the projection
  map. We claim that $V_x$ and $V_y$ are as needed. Let
  $c_1\in V_x$ and $c_2\in V_y$ be arbitrary extreme
  isolated points such that for some $\lambda\in[0,1]$ the
  point $\lambda c_1+(1-\lambda)c_2$ is a limit of isolated
  extreme points of $\Phi(X,d_X)$. By the remarks in the
  previous paragraph, there is $p\in\Sch(X,d_X,D)$ such that
  $p=(d_1,d_2,V)$ for some $d_1,d_2\in D$ with
  $\pi(c_1)=d_1$ and $\pi(c_2)=d_2$ and $V$ is a basic open
  neighborhood of $d_X(d_1,d_2)$. Now, since $c_1\in V_x$
  and $c_2\in V_y$ we have that
  $d_X(d_1,d_2)\in(\alpha-\varepsilon,\alpha+\varepsilon)\subseteq
  U$, so $V\cap U$ is nonempty. Pick any $\lambda\in V\cap
  U$ and note that $\lambda c_1+(1-\lambda)c_2$ is also a
  limit of isolated extreme points of $\Phi(X,d_X)$ since
  $\lambda\in V$.

  We also need to show that $\alpha$ is the only number is
  $[0,1]$ satisfying (\ref{eq:last}) for $x$ and $y$. Pick
  any $\beta\in[0,1]$ distinct from $\alpha$. Let
  $\varepsilon>0$ be smaller than $|\alpha-\beta|$. Pick a
  basic open neighborhood
  $U\subseteq(\beta-\varepsilon\slash
  2,\beta+\varepsilon\slash 2)$ of $\beta$. We claim that
  $U$ witnesses that (\ref{eq:last}) is not satisfied. Let
  $V_x$ and $V_y$ be arbitrary open neighborhoods of $x$ and
  $y$, respectively, in $\Phi(X,d_X)$. Pick $d_1,d_2\in D$
  such that $d_X(d_1,x)<\varepsilon\slash 2$,
  $d_X(d_2,y)<\varepsilon\slash 2$ and $d_1\in V_x$, $d_2\in
  V_y$. Note that $d_X(d_1,d_2)\in(\alpha-\varepsilon\slash
  2,\alpha+\varepsilon\slash 2)$ and that $U$ is disjoint
  from $(\alpha-\varepsilon\slash 2,\alpha+\varepsilon\slash
  2)$. Find a basic open neighborhood $V$ of $d_X(d_1,d_2)$
  in $[0,1]$ such that $\overline V\cap U=\emptyset$. Find
  $n\in\N$ big enough so that $p_n\in\Sch(X,d_X,D)$ is equal
  to $(d_1,d_2,V)$ and such that letting $c_1=c_1(p_n)$,
  $c_2=c_2(p_n)$ we have $c_1\in V_x$ and $c_2\in V_y$. Now,
  there is $\lambda\in[0,1]$ such that $\lambda
  c_1+(1-\lambda)c_2$ is a limit of isolated extreme points
  of $\Phi(X,d_X)$ but the set of such $\lambda$ is equal to
  $\overline V$, which is disjoint from $U$. This shows that
  $\beta$ does not satisfy (\ref{eq:last}) and ends the
  proof.
\end{proof}

Now we are ready to finish the proof that $\Phi$ is a
reduction.

\begin{proposition}\label{final}
  If $(X,d_X)$ and $(Y,d_Y)$ are perfect closed subspaces of
  $\U_1$ and $\Phi(X,d_X)$ is affinely homeomorphic to
  $\Phi(Y,d_Y)$, then $(X,d_X)$ is isometric to $(Y,d_Y)$.
\end{proposition}

\begin{proof}
  Let $\varphi:\Phi(X,d_X)\to\Phi(Y,d_Y)$ be an affine
  homeomorphism. Note that $\varphi$ maps nonisolated
  extreme points of $\Phi(X,d_X)$ to nonisolated extreme
  points of $\Phi(Y,d_Y)$, so, by Claim \ref{nonisolated},
  $\varphi''\ext(S(X,d_X))=\ext(S(Y,d_Y))$. Now, $X\subseteq
  S(X,d_X)$ and $Y\subseteq S(Y,d_Y)$ are dense $G_\delta$
  sets, so there is a comeager set $X'\subseteq X$ such that
  $Y'=\varphi'' X'$ is comeager in $Y$. Note that since
  $\varphi$ preserves the topological and affine structure,
  (\ref{eq:last}) is preserved by $\varphi$ and so Lemma
  \ref{logic} implies that for $x_1,x_2\in X'$ we have
  $d_Y(\varphi(x_1),\varphi(x_2)=d_X(x_1,x_2)$. This means
  that $(X',d_X)$ and $(Y',d_Y)$ are isometric. Since $X'$
  is comeager in $X$ and $Y'$ is comeager in $Y$, the spaces
  $X$ and $Y$ are isometric as well.
\end{proof}

Theorem \ref{main} now follows from Propositions
\ref{isomperfect}, \ref{homo} and \ref{final}.

\section{Concluding remarks and open questions}
\label{sec:questions}

A major problem \cite[Problem 5.2]{clemens.gao.kechris} that
is still left open is that of the complexity of the
homeomorphism relation for compact metric spaces. As shown
by Kechris and Solecki (see also \cite[Theorem 1.4]{ftt} for
a new proof of this result), it is an orbit equivalence
relation that is bi-reducible with an action of the group of
homeomorphisms of the Hilbert cube. Homeomorphism of compact
metric spaces is Borel reducible to the affine homeomorphism
of Bauer simplices and it is not known whether it is also a
complete orbit equivalence relation. As a comment to this
problem, let us note that from the point of view of Banach
space theory, this is the question whether the isometry
problems for general separable Banach spaces and for
separable Banach spaces of the form $C(K)$, have the same
complexities.

In \cite[Theorem 7.3]{ftt} Farah, Toms and T\"ornquist
showed that the isomorphism of separable simple nuclear
C*-algebras is below an action of the automorphism group of
the Cuntz algebra $\textrm{Aut}(\mathcal{O}_2)$. Since
complete orbit equivalence relations are typically induced
by actions of universal Polish groups, the following
question is natural.

\begin{question}
  Is the group $\textrm{Aut}(\mathcal{O}_2)$ a universal
  Polish group?
\end{question}

\bibliographystyle{plain}
\bibliography{refs}

\end{document}